\documentclass[10pt]{amsart}
\usepackage{amssymb,bm,amsmath,mathrsfs, MnSymbol}

\newcommand{\map}[1]{\xrightarrow{#1}}
\newcommand{\iso}{\cong}

\newcommand{\Hom}{\mathrm{Hom}}
\newcommand{\Aut}{\mathrm{Aut}}
\newcommand{\End}{\mathrm{End}}
\newcommand{\Spec}{\mathrm{Spec}}
\newcommand{\Q}{\mathbb Q}
\newcommand{\Z}{\mathbb Z}
\newcommand{\R}{\mathbb R}
\newcommand{\C}{\mathbb C}
\newcommand{\F}{\mathbb F}
\newcommand{\A}{\mathbb A}
\newcommand{\co}{\mathcal O}
\newcommand{\FF}{{\mathcal{F}}}
\newcommand{\LL}{{\mathcal{L}}}
\newcommand{\VV}{\mathcal{V}}
\newcommand{\UU}{\mathcal{U}}
\renewcommand{\SS}{\mathcal{S}}

\newcommand{\ord}{\mathrm{ord}}
\newcommand{\length}{\mathrm{length}}
\newcommand{\Lie}{\mathrm{Lie}}

\newcommand{\inv}{\mathrm{inv}}
\newcommand{\GSpin}{\mathrm{GSpin}}
\newcommand{\Spin}{\mathrm{Spin}}
\newcommand{\SO}{\mathrm{SO}}

\newcommand{\GU}{\mathrm{GU}}
\newcommand{\GL}{\mathrm{GL}}
\newcommand{\SL}{\mathrm{SL}}
\newcommand{\action}{\bullet}

\usepackage{color}

\begin{document}
\author{Benjamin Howard}
\author{George Pappas}
 \thanks{G.P. is partially supported by NSF grant  DMS-1102208}
\thanks{B.H. is partially supported by NSF grant  DMS-1201480}
\address{Dept.~of Mathematics\\ Boston College\\
Chestnut Hill\\ MA 02467-3806}
\address{Dept.~of Mathematics\\ Michigan State University\\ E. Lansing\\ MI 48824-1027}

\title{On the supersingular locus of the GU(2,2) Shimura variety}
\date{\today}

\begin{abstract}
We describe the supersingular locus of a $\GU(2,2)$ Shimura variety at a prime inert
in the corresponding quadratic imaginary field.
\end{abstract}

\maketitle

\theoremstyle{plain}
\newtheorem{theorem}{Theorem}[section]
\newtheorem{proposition}[theorem]{Proposition}
\newtheorem{lemma}[theorem]{Lemma}
\newtheorem{corollary}[theorem]{Corollary}
\newtheorem{conjecture}[theorem]{Conjecture}
\newtheorem{BigThm}{Theorem}

\theoremstyle{definition}
\newtheorem{definition}[theorem]{Definition}
\newtheorem{hypothesis}[theorem]{Hypothesis}

\theoremstyle{remark}
\newtheorem{remark}[theorem]{Remark}
\newtheorem{question}[theorem]{Question}

\numberwithin{equation}{section}
\renewcommand{\theBigThm}{\Alph{BigThm}}

\section{Introduction}

This paper contributes to the theory of integral models of Shimura varieties,
and, in particular, to the problem of explicitly describing the basic locus in the reduction modulo $p$
of a canonical integral model.
In many cases where this integral model is a moduli space of abelian varieties with additional structures,
the basic locus coincides with the supersingular locus, \emph{i.e.}~with the subset
of the moduli in positive characteristic where the corresponding abelian variety is isogenous to a product
of supersingular elliptic curves. The first investigations of a higher dimensional supersingular locus were
for the Siegel moduli space, and are due to Koblitz, Katsura-Oort, and Li-Oort.  See the introduction
of \cite{Vol} for these and other references.
More recently, such explicit descriptions for certain unitary and orthogonal Shimura varieties have found  
applications to Kudla's program  relating arithmetic intersection numbers of special cycles on Shimura varieties to
Eisenstein series; this motivated further study, as in \cite{KR, KR1, KR2, KRlocal} and \cite{VolWed}.

In this paper, we study the supersingular locus of the special fiber of  a $\GU(2,2)$ 
Shimura variety  at an odd prime  inert in the corresponding imaginary quadratic field. 
Our methods borrow liberally from Vollaard \cite{Vol} and Vollaard-Wedhorn \cite{VolWed}, who considered the 
 $\GU(n, 1)$ Shimura varieties at  inert primes, and Rapoport-Terstiege-Wilson \cite{RTW}, who considered
the  $\GU(n, 1)$ Shimura varieties at ramified primes.   
If one attempts to directly imitate the arguments of \cite{RTW,Vol,VolWed} 
to study the general $\GU(r,s)$ Shimura variety, the method breaks down 
at a crucial point.  The key new idea for overcoming this 
obstacle is to exploit the linear algebra  underlying a twisted version of the exceptional isomorphism   
$\mathrm{SU}(2,2)\iso \Spin(4,2)$ corresponding to the Dynkin diagram identity $A_3= D_3$.   
As such,  we do not expect our methods to extend to unitary groups
of other signatures (although we do hope that our result
will eventually help to predict the shape of the answer in the general case).  
The problem of understanding the supersingular locus of the 
$\GU(3,2)$ Shimura variety, for example, remains open.

However,  our methods should   extend to the family of 
$\mathrm{GSpin}(n,2)$ Shimura varieties.  Work of Kisin \cite{Kisin} and Madapusi Pera \cite{MP} 
(see also the papers of Vasiu)   provides us with 
 a good theory of integral models for these Shimura varieties, and  recent work of  
 W.~Kim \cite{Kim} gives a good theory of 
 Rapoport-Zink spaces  as well.  An extension of our results in this direction would have applications
to Kudla's program, for example by allowing one to generalize the work of 
Kudla-Rapoport  \cite{KR1,KR2} from  $\GSpin(2,2)$ and $\GSpin(3,2)$ Shimura varieties to the general
$\GSpin(n,2)$ case.  Using the isomorphism between $\GSpin(6,2)$ and the similitude group  
of a four-dimensional symplectic module over the  Hamiltonian quaternions \cite{FH}, one could also expect to 
generalize B\"ultel's  results \cite{Bul} on the  supersingular locus of the 
moduli space of polarized abelian eightfolds
with an action of a definite quaternion algebra.  More ambitiously, one could hope to exploit 
the connection between polarized K3 surfaces and the $\mathrm{GSpin}(19,2)$ Shimura variety in order to 
study the moduli space of supersingular K3 surfaces.  Some of these topics will be pursued
in subsequent papers.

As this paper was being prepared, G\"ortz and He were conducting  a general study of 
basic minuscule affine Deligne-Lusztig varieties  for equicharacteristic discrete valued fields.  
The preprint  \cite{GH} provides a list of cases where these affine Deligne-Lusztig varieties  can be expressed as a union 
of usual Deligne-Lusztig varieties, and that list contains an equicharacteristic analogue of the  $\GU(2,2)$
Rapoport-Zink space considered here.  These results  of  G\"ortz and He  in the equicharacteristic case
are analogous to our  mixed characteristic results.

The authors would like to thank M. Rapoport for many useful suggestions,
and X. He for communicating his joint results with U. G\"ortz.
 
\subsection{The local result}

Our main result concerns the structure of the Rapoport-Zink space parametrizing quasi-isogenies
between certain $p$-divisible groups with extra structure.  
Fix an algebraically closed field $k$ of characteristic $p>2$, let $W$ be the 
ring of Witt vectors over $k$, and let $E/\Q_p$ be an unramified 
degree two extension.  Consider the family of triples $(G,\iota,\lambda)$,
defined over $W$-schemes $S$ on which $p$ is locally nilpotent, consisting of a 
supersingular $p$-divisible group $G$ with an action $\iota:\co_E\to \End(G)$
and a principal polarization $\lambda : G\to G^\vee$.  We require that the action $\iota$ and the polarization
$\lambda$ be compatible in the sense of (\ref{linear polarization}), and that the action of $\co_E$
on $\Lie(G)$ satisfy the signature $(2,2)$ determinant condition of (\ref{kottwitz}).
A choice of one such triple $(\bm{G},\bm{\iota},\bm{\lambda})$ over $k$ as a basepoint
determines the \emph{Rapoport-Zink space}, $\mathscr{M}$, parametrizing 
quadruples $(G,\iota,\lambda,\varrho)$ in which 
$\varrho: G\times_S S_0 \to \bm{G}\times_k S_0$ is an $\co_E$-linear quasi-isogeny  
under which $\bm{\lambda}$ pulls back to a $\Q_p^\times$ multiple $c(\varrho) \lambda$. 
Here $S_0=S\times_W k$.   The Rapoport-Zink space $\mathscr{M}$ is a formal scheme 
over $W$, and admits a decomposition into open and closed formal
subschemes $\mathscr{M}=\biguplus_{\ell\in \Z} \mathscr{M}^{(\ell)}$,
where $\mathscr{M}^{(\ell)}$ is the locus where $\ord_p(c(\varrho)) = \ell$.
Here and elsewhere, we use the symbol $\biguplus$ to denote disjoint union.
The group $p^\Z$ acts on $\mathscr{M}$, where the action of $p$ sends 
$(G,\iota,\lambda,\varrho)\mapsto (G,\iota,\lambda, p\varrho)$.
This action has $\mathscr{M}^{(0)} \uplus \mathscr{M}^{(1)}$ as a fundamental domain.
In fact the action of $p^\Z$ extends to a larger group $J$ which acts transitively on the set
$\{ \mathscr{M}^{(\ell)} : \ell\in \Z\}$.
Define  
\[
\mathscr{N} = p^\Z \backslash \mathscr{M}
\]
and let $\mathscr{N}^+$ and $\mathscr{N}^-$
be the images of  $\mathscr{M}^{(0)}$ and $\mathscr{M}^{(1)}$, respectively, under the 
quotient map $\mathscr{M} \to \mathscr{N}$.

  In Section \ref{s:moduli} we construct a $6$-dimensional  $\Q_p$-vector space 
\[
\bm{L}_\Q^\Phi \subset \End( \bm{G} )_\Q
\] 
of \emph{special quasi-endomorphisms} of $\bm{G}$ as the $\Phi$-fixed vectors in a slope
$0$ isocrystal $(\bm{L}_\Q,\Phi)$.  The vector space $\bm{L}_\Q^\Phi$ is endowed with a 
$\Q_p$-valued quadratic  form $Q(x) =x\circ x$, and we define a \emph{vertex lattice} in 
$\bm{L}_\Q^\Phi$ to be a $\Z_p$-lattice $\Lambda\subset \bm{L}_\Q^\Phi$ such that 
\[
p\Lambda\subset \Lambda^\vee \subset \Lambda.
\]
The \emph{type} $t_\Lambda \in \{2,4,6\}$ of $\Lambda$ is the dimension of $\Lambda/\Lambda^\vee$.
To each   point $(G,\iota,\lambda,\varrho)$ of $\mathscr{N}$, the quasi-isogeny $\varrho$
allows us to view $\Lambda$ as a lattice of quasi-endomorphisms of $G$.  Let 
$\widetilde{\mathscr{N}}_\Lambda \subset\mathscr{N}$ be the locus of points where 
 $\Lambda\subset \End(G)$  (\emph{i.e.}~the locus where these quasi-endomorphisms are integral).  It is a closed formal subscheme 
of $\mathscr{N}$, whose underlying reduced $k$-scheme we denote by $\mathscr{N}_\Lambda$.
We show that the underlying reduced subscheme $\mathscr{N}_\mathrm{red}$ of $\mathscr{N}$
is covered by these closed subschemes:
\[
\mathscr{N}_\mathrm{red}= \bigcup_\Lambda \mathscr{N}_\Lambda,
\]
and that 
\[
\mathscr{N}_{\Lambda_1} \cap \mathscr{N}_{\Lambda_2} = 
\begin{cases}
\mathscr{N}_{ \Lambda_1 \cap \Lambda_2} & \mbox{if $\Lambda_1\cap \Lambda_2$ is a vertex lattice} \\
\emptyset & \mbox{otherwise}
\end{cases}
\]
where the left hand side is understood to mean the reduced subscheme underlying the scheme-theoretic 
intersection (we suspect that the scheme-theoretic intersection is already reduced, but are unable to provide a proof).

Section \ref{strata} is devoted to understanding the structure of 
$\mathscr{N}_\Lambda^\pm = \mathscr{N}_\Lambda \cap \mathscr{N}^\pm$.
Setting $d_\Lambda=t_\Lambda/2$, we prove that  $\mathscr{N}_\Lambda^\pm$ 
is a projective, smooth, and irreducible $k$-scheme of dimension $d_\Lambda-1$.  In fact,
\begin{enumerate}
\item
if $d_\Lambda=1$ then $\mathscr{N}_\Lambda^\pm$ is  single point,
\item
if $d_\Lambda=2$ then $\mathscr{N}_\Lambda^\pm$ is isomorphic to  $\mathbb{P}^1$,
\item
if $d_\Lambda=3$ then $\mathscr{N}_\Lambda^\pm$ is isomorphic to the Fermat hypersurface 
\[
x_0^{p+1}+x_1^{p+1}+x_2^{p+1}+x_3^{p+1} = 0.
\]
\end{enumerate}
The irreducible components of $\mathscr{N}^\pm$ are precisely the closed 
subschemes $\mathscr{N}^\pm_\Lambda$ indexed by the type $6$ vertex lattices.
From  this we deduce the following theorem.

\begin{BigThm}
The underlying reduced scheme $\mathscr{M}^{(\ell)}_\mathrm{red}$ of $\mathscr{M}^{(\ell)}$ is connected.
Every irreducible component of $\mathscr{M}^{(\ell)}_\mathrm{red}$ is a smooth $k$-scheme of 
dimension $2$, isomorphic to the Fermat hypersurface
\[
x_0^{p+1}+x_1^{p+1}+x_2^{p+1}+x_3^{p+1}=0.
\]
If two irreducible components intersect nontrivially,  the reduced scheme underlying their 
scheme-theoretic intersection is either a point or a projective line.  
\end{BigThm}

See  Sections \ref{ss:main} and \ref{ss:building2} for a more detailed description of $\mathscr{M}_\mathrm{red}$.

\subsection{The global result}
 
 In Section \ref{s:global} we consider the global situation.  Let $E$ be a quadratic imaginary field,
 and let $p>2$ be inert in $E$. Let $\co\subset E$ be the integral closure of $\Z_{(p)}$,
 and let $V$ be a free $\co$-module of rank $4$ endowed with a perfect $\co$-valued
 Hermitian form of signature $(2,2)$.  Let  $G=\GU(V)$ be the group of unitary similitudes
of $V$,  a reductive group over $\Z_{(p)}$.   Fix a compact open subgroup $U^p \subset G(\A_f^p)$, 
which we assume is sufficiently small, and define
$U_p=G(\Z_p)$ and $U=U_pU^p \subset G(\A_f)$.

Using this data we define a scheme $M_U$, smooth of relative dimension $4$ over $\Z_{(p)}$,
as a moduli space of abelian fourfolds, up to prime-to-$p$-isogeny, with additional structure, in such a way that
the complex fiber of $M_U$ is the Shimura variety
\[
M_U(\C) = G(\Q) \backslash  (  \mathcal{D} \times G(\A_f) / U ).
\]
Here $\mathcal{D}$ is the Grassmannian of negative definite planes in $V\otimes_{\co}  \C$.

 Let $k$ be an algebraically closed  field of characteristic $p$,  and denote by $M_U^\mathrm{ss}$  
 the reduced supersingular locus  of the geometric special fiber $M_U \times_{\Z_{(p)} }k$. 
 The uniformization theorem of Rapoport and Zink expresses $M_U^\mathrm{ss}$
 as a disjoint union of quotients of the scheme $\mathscr{M}_\mathrm{red}$ described above.
As a consequence we obtain the following result.

 \begin{BigThm}
The $k$-scheme $M_U^\mathrm{ss}$ has pure dimension $2$.  
For $U^p$ sufficiently small, all irreducible components of 
$M_U^\mathrm{ss}$ are isomorphic to the Fermat hypersurface 
\[
x_0^{p+1}+x_1^{p+1}+x_2^{p+1} + x_3^{p+1}=0.
\]
If two irreducible components intersect nontrivially,  the reduced scheme underlying their 
scheme-theoretic intersection is either a point or a projective line.  
\end{BigThm}

 
 \subsection{Notation}
 
 
We use the following notation throughout Sections \ref{s:moduli} and \ref{strata}.
Fix an odd  prime $p$ and an unramified quadratic extension $E$ of the field of $p$-adic numbers $\Q_p$.  
The nontrivial Galois automorphism of $E$ is denoted
$\alpha\mapsto \overline{\alpha}$.  Let $k$ be an algebraically closed field of characteristic $p$.  Its ring
of Witt vectors $W=W(k)$ is a complete discrete valuation ring with residue field $k=W/pW$ and fraction
field $W_\Q$.  Label the two embeddings of $\co_E$ into $W$ as
\[
\psi_0: \co_E \to W\qquad \psi_1:\co_E\to W,
\]
and denote by $\sigma$ both the absolute Frobenius $x\mapsto x^p$ on $k$ and 
its unique lift to a ring automorphism of $W$.    Denote by 
$
\epsilon_0,\epsilon_1\in \co_E\otimes W
$ 
the orthogonal idempotents characterized by 
\[
\epsilon_i M =\{ x\in M : (\alpha \otimes 1) \cdot x = ( 1 \otimes \psi_i(\alpha) )  \cdot x , \ \forall \alpha\in \co_E \}
\]
for any $\co_E \otimes W$-module $M$.  For any $\Z$-module $M$ we abbreviate
 $M_\Q=M\otimes_\Z\Q$.  In particular, $M_\Q=M\otimes_W W[1/p]$ for any $W$-module $M$.


\section{Moduli spaces and lattices}
\label{s:moduli}


In this section we recall the Rapoport-Zink space of a $\GU(2,2)$ Shimura variety, and define a
 stratification of the underlying reduced scheme.


\subsection{The Rapoport-Zink space}
\label{ss:RZ}


Let $\mathrm{Nilp}_W$ be the category of $W$-schemes on which $p$ is locally nilpotent.
We wish to parametrize  triples $(G,\iota,\lambda )$ over objects  $S$ of $\mathrm{Nilp}_W$ in which 
\begin{itemize}
\item
$G$  is a supersingular $p$-divisible group  of dimension $4$,
 \item
$\iota :\co_E \to \End(G)$ is an action of $\co_E$ on $G$,
\item 
 $\lambda:G \to G^\vee$ is a principal polarization.
\end{itemize}
We further require that every $\alpha \in \co_E$ satisfies both the \emph{$\co_E$-linearity condition}
\begin{equation}\label{linear polarization}
\lambda\circ \iota (\overline{\alpha} )  = \iota(\alpha)^\vee \circ \lambda ,
\end{equation}
 and the  \emph{signature $(2,2)$-condition}
\begin{equation}\label{kottwitz}
\det(T- \iota(\alpha) ; \Lie(G) ) = (T-\psi_0(\alpha))^2(T-\psi_1(\alpha))^2
\end{equation}
as sections of $\co_S[T]$.  The signature $(2,2)$ condition is equivalent to
each of the $\co_S$-module direct summands in 
$
\Lie(G) =\epsilon_0\Lie(G) \oplus \epsilon_1\Lie(G)
$
being locally free of rank $2$.

Fix one such triple $(\bm{G},\bm{\iota},\bm{\lambda})$ over $k$ as a base point, 
and let $\mathscr{M}$ be the functor on $\mathrm{Nilp}_W$
sending $S$ to the set of isomorphism classes of quadruples $(G,\lambda,i,\varrho)$ 
over $S$ where $(G, \iota, \lambda )$ is as above, and 
$
\varrho: G_{ /S_0} \to \bm{G}_{/S_0} 
$
is an $\co_E$-linear quasi-isogeny such that $\varrho^* \bm{\lambda} = c(\varrho) \lambda$ 
for some $c(\varrho) \in \Q_p^\times$.   Here $S_0$ is the $k$-scheme $S\otimes_W k$.    
The functor $\mathscr{M}$ is represented by a formal scheme locally of finite type over 
$\mathrm{Spf}(W)$ by \cite{RZ}.  There is a decomposition 
$\mathscr{M}=\biguplus_{\ell \in \Z} \mathscr{M}^{(\ell)}$ into open and closed formal subschemes, 
where $\mathscr{M}^{(\ell)}$ is the locus of points where $\ord_p (  c(\varrho) ) = \ell$.

Let  $J \subset \End(\bm{G})_\Q^\times$ denote the subgroup of $E$-linear elements such that 
$g^*\bm{\lambda} = \nu(g)\bm{\lambda}$ for some $\nu(g)\in \Q_p^\times$. The group $J$ acts on 
$\mathscr{M}$ in an obvious way:
\[
g\cdot (G, \iota, \lambda, \varrho) = (G, \iota, \lambda, g\circ \varrho).
\]
As usual, the group $J$ is  the $\Q_p$-points of a reductive group  over $\Q_p$.
In fact, by  \cite[Remark 1.16]{Vol}  this reductive group  is  the   group of unitary similitudes of
 the split Hermitian space of dimension  $4$ over $E$.  In particular the derived subgroup 
 $J^\mathrm{der}$ is isomorphic to the special unitary  group, and the similitude character 
 $\nu:J\to \Q_p^\times$ is surjective.   Note that the action of any $g \in J$ with $\ord_p(\nu(g))=1$ defines an 
isomorphism $ \mathscr{M}^{(\ell)} \iso \mathscr{M}^{(\ell+1)}.$

As a special case of this action,  the group $p^\Z$ acts on $\mathscr{M}$ by $p\cdot (G,\iota, \lambda,\varrho) = (G, \iota, \lambda , p\varrho)$, 
and  the quotient  $\mathscr{N} = p^\Z \backslash \mathscr{M}$
has $\mathscr{M}^{(0)}\uplus \mathscr{M}^{(1)}$ as a fundamental domain.
Let $\mathscr{N}^+\iso \mathscr{M}^{(0)}$ and $\mathscr{N}^-\iso \mathscr{M}^{(1)}$ 
be the open and closed formal subschemes of
$\mathscr{N}$ on which $\ord_p(c(\varrho))$ is even and odd,  respectively.  By the previous paragraph there is an isomorphism
$\mathscr{N}^+\iso \mathscr{N}^-$, and we will see later  in 
Theorem \ref{thm:connected} that $\mathscr{N}^+$ and $\mathscr{N}^-$ are 
precisely the connected components of $\mathscr{N}$.


\subsection{Special endomorphisms}


 In this subsection we will define a $\Q_p$-subspace
\[
\bm{L}_\Q^\Phi \subset \End(\bm{G})_\Q
\]
of \emph{special quasi-endomorphisms} of $\bm{G}$ in such a way that 
$x\mapsto x\circ x$ defines a $\Q_p$-valued quadratic form on $\bm{L}_\Q^\Phi$.
The subspace $\bm{L}_\Q^\Phi$ is not quite canonical; it will depend on the auxiliary choice of a 
certain tensor $\bm{\omega}$ in the top exterior power of the Dieudonn\'e module of $\bm{G}$.

Denote by $\bm{D}$ the covariant Dieudonn\'e module of $\bm{G}$, with its induced action of $\co_E$
and induced  alternating form 
$
\lambda : \bigwedge_W^2\bm{D} \to W
$
satisfying $\lambda( Fx,  y ) =  \lambda( x, V y )^\sigma.$
Under the covariant conventions, $\Lie(\bm{G}) \iso \bm{D}/V\bm{D}
$ as $k$-vector spaces with $\co_E$-actions. 
Abbreviate $\bigwedge^\ell_E \bm{D} = \bigwedge^\ell_{\co_E \otimes W} \bm{D}.$   
Once we fix a  $\delta\in \co_E^\times$ 
satisfying $\delta^\sigma =-\delta$,  there is a unique Hermitian form
\[
\langle\cdot,\cdot \rangle: \bm{D} \times \bm{D} \to \co_E \otimes  W
\]
satisfying 
\begin{equation}\label{hermite}
\lambda (x,y) = \mathrm{Tr}_{E/\Q_p}  \delta^{-1}\langle    x,y \rangle  ,
\end{equation}
which in turn induces a  Hermitian form  on every exterior power $\bigwedge^\ell_E \bm{D}$ by
\[
\langle  x_1\wedge \cdots \wedge x_\ell , y_1\wedge\cdots \wedge y_\ell \rangle
= \sum_{\pi \in S_\ell } \mathrm{sgn}(\pi)  \prod_{i=1}^\ell  \langle x_i, y_{\pi(i)} \rangle.
\]
This Hermitian form identifies each lattice $\bigwedge_E^\ell \bm{D}$ with its dual lattice in $(\bigwedge_E^\ell \bm{D})_\Q$.

In order to make explicit calculations, we now put coordinates on $\bm{D}_\Q$.

\begin{lemma}\label{lem:standard coords}
There are $W_\Q$-bases 
\begin{align*}
e_1, e_2,e_3 ,e_4 & \in \epsilon_0\bm{D}_\Q \\
f_1, f_2,f_3,f_4  &  \in \epsilon_1\bm{D}_\Q
\end{align*} 
such that
\begin{equation}\label{standard Hermitian}
\langle e_i,f_j\rangle = \begin{cases}
\epsilon_0 &\mbox{if $i=j$} \\
0 & \mbox{otherwise,}
\end{cases}
\end{equation}
and the $\sigma$-semi-linear operator $F$ satisfies
\[
\begin{array}{cccc}
Fe_1=f_1, & Fe_2=f_2, &  Fe_3=pf_3, &  Fe_4=pf_4 \\
Ff_1=pe_1, & Ff_2=pe_2, &  Ff_3=e_3, &  Ff_4=e_4 .
\end{array}
\]
\end{lemma}

\begin{proof}
Denote by $\bm{D}'_\Q$ the isocrystal with $W_\Q$-basis $\{e_1, \ldots, e_4,f_1, \ldots, f_4\}$ and $F$ operator
defined by the above relations.  Endow $\bm{D}'_\Q$ with the $E$-action $\bm{\iota}'(\alpha)e_i = \psi_0(\alpha)e_i$ and
$\bm{\iota}'(\alpha)f_i = \psi_1(\alpha) f_i$, and the unique Hermitian form satisfying (\ref{standard Hermitian}).  This 
Hermitian form determines a polarization $\bm{\lambda}'(x,y) = \mathrm{Tr}_{E/\Q_p} \delta^{-1}\langle x,y\rangle$.
As $\bm{D}'_\Q$ is isoclinic of slope $1/2$, there is an isomorphism of isocrystals 
\[
\varrho:\bm{D}_\Q\iso \bm{D}'_\Q.
\]
Any two embeddings of  $E$ into  $\End( \bm{D}'_\Q )$ are conjugate, by the Noether-Skolem theorem, and so $\varrho$ 
may be modified to make it $E$-linear. Another application of Noether-Skolem shows that $\varrho$ may be further modified 
to ensure that the polarizations on $\bm{D}_\Q$ and $\bm{D}'_\Q$  induce the same Rosati involution on 
\[  
\End( \bm{D}_\Q ) \iso \End( \bm{D}'_\Q ). 
\]  
This implies that $\varrho$  identifies the polarizations, and hence the  Hermitian forms, on $\bm{D}_\Q$ and $\bm{D}'_\Q$ 
up to scaling by an element $c(\varrho)\in \Q_p^\times$. 

Finally, for every $c\in \Q_p^\times$ one can find an $E$-linear isocrystal automorphism $g$ of 
$\bm{D}'_\Q$  such that  $g$  rescales the polarization of $\bm{D}'_\Q$ by the factor  $c$.  
For example, if $\ord_p(c)$ is even then write $c=\alpha\overline{\alpha}$
with $\alpha\in E^\times$ and take $g=\bm{\iota}'(\alpha)$.  If $c=p$ then take $g$ to be
\[
\begin{array}{cccc}
e_1\mapsto e_3 , & e_2 \mapsto e_4, &  e_3\mapsto pe_1, &  e_4\mapsto pe_2 \\
f_1\mapsto pf_3, & f_2 \mapsto pf_4, &  f_3 \mapsto f_1 , &  f_4 \mapsto f_2 .
\end{array}
\]
Thus $\varrho$ may be further modified to ensure that $c(\varrho)=1$.
\end{proof}

\begin{lemma}\label{lem:omega}
There is an $\co_E\otimes W$-module generator  $\bm{\omega}\in \bigwedge_{E}^4 \bm{D}$ such 
 that $\langle \bm{\omega},\bm{\omega} \rangle=1$, and $F\bm{\omega} = p^2\bm{\omega}$.  
  If $\bm{\omega} ' \in \bigwedge_{E}^4 \bm{D}$ is another such element, there is an $\alpha\in \co_E^\times$ such that
$\alpha \overline{\alpha}=1$ and $\bm{\omega}'=\alpha \bm{\omega}$.
\end{lemma}

\begin{proof}
The $W$-module decomposition $\bm{D}=\epsilon_0\bm{D} \oplus \epsilon_1\bm{D}$ induces  a corresponding  decomposition 
$
\bigwedge^4_E \bm{D} = \bigwedge^4 \epsilon_0\bm{D} \oplus \bigwedge^4 \epsilon_1\bm{D}.
$
Fixing a basis as in Lemma \ref{lem:standard coords}, we must have
\begin{align*}
\bigwedge\nolimits^4 \epsilon_0 \bm{D} &=W \cdot p^{k_0} e_1\wedge e_2\wedge e_3\wedge e_4 \\
\bigwedge\nolimits^4 \epsilon_1 \bm{D} &= W\cdot p^{k_1} f_1\wedge f_2 \wedge f_3 \wedge f_4 
\end{align*}
for some integers $k_0$ and $k_1$.   
 The self-duality of $\bigwedge^4_E \bm{D}$ under $\langle\cdot,\cdot\rangle$  implies $k_0+k_1=0$.
The signature $(2,2)$ condition on 
\[
\Lie(\bm{G}) \iso \bm{D}/V\bm{D} =  \epsilon_0\bm{D}/V \epsilon_1\bm{D}  \oplus \epsilon_1 \bm{D}_1/V \epsilon_0\bm{D} 
\]
implies that each of the summands on the right has dimension $2$ over $W/pW$, and hence the 
cokernels of 
\[
V: \bigwedge\nolimits^4 \epsilon_0\bm{D} \to \bigwedge\nolimits^4  \epsilon_1\bm{D}, \qquad
V: \bigwedge\nolimits^4 \epsilon_1\bm{D} \to \bigwedge\nolimits^4 \epsilon_0\bm{D}
\]
are each of length $2$ as $W$-modules.  Using 
\begin{align*}
V(e_1\wedge e_2\wedge e_3\wedge e_4) & = p^2 f_1\wedge f_2 \wedge f_3 \wedge f_4 \\
V( f_1\wedge f_2 \wedge f_3 \wedge f_4 ) & = p^2 e_1\wedge e_2\wedge e_3\wedge e_4
\end{align*}
we deduce that $k_1$ and $k_2$ are equal, and hence both are equal to $0$.  It follows that
\begin{equation}\label{standard omega}
\bm{\omega} = e_1\wedge e_2\wedge e_3\wedge e_4 + f_1\wedge f_2 \wedge f_3 \wedge f_4 
\end{equation}
generates  $\bigwedge_{E}^4 \bm{D}$ as an $\co_E\otimes W$-module.  A simple calculation shows that 
 $\langle \bm{\omega},\bm{\omega} \rangle=1$ and $F\bm{\omega} = p^2\bm{\omega}$,
proving the existence part of the lemma.  The uniqueness part of the claim is obvious.
\end{proof}

\begin{definition}
For any  $\bm{\omega}$ as in the lemma,  define the  \emph{Hodge star operator}  
$x\mapsto x^\star$ on $\bigwedge^2_E \bm{D}$ by the relation 
$
y\wedge  x^\star = \langle y,x\rangle   \cdot  \bm{\omega}
$  
for all $y\in \bigwedge^2_E \bm{D}$.     
\end{definition}

The Hodge operator satisfies $(\alpha x)^\star = \overline{\alpha}x^\star$
for all $\alpha\in \co_E\otimes W$.  Denote by 
\[
\bm{L} = \{ x \in   \bigwedge\nolimits_E ^2 \bm{D} : x^\star=x \}
\]
the $W$-submodule of Hodge fixed vectors. 
The  Hermitian form $\langle\cdot,\cdot\rangle$  on $\bm{D}$ determines an injection
$\bigwedge^2_E \bm{D} \to \End_W( \bm{D} )$ 
by  
\[
( a\wedge b) (z) = \langle a,z\rangle b - \langle b,z\rangle a,
\]
and we obtain inclusions
$
\bm{L} \subset \bigwedge^2_E \bm{D} \subset \End_W(\bm{D}).
$
Note that both the Hodge star operator and  the submodule $\bm{L}$  depend on the choice of $\bm{\omega}$.

\begin{proposition}\label{prop:basic hodge}
For any choice of $\bm{\omega}$, the induced Hodge star operator has the following properties.
\begin{enumerate}
\item
Every  $x\in \bigwedge^2_E\bm{D}$ satisfies $(x^\star)^\star=x$.
 \item
Every $x\in \bm{L}$, viewed as an endomorphism of $\bm{D}$, satisfies 
\begin{equation}\label{quadratic calc}
x\circ x = -  \frac{\langle x,x\rangle }{2}  .
\end{equation}
In particular,  $Q(x)= x\circ x$ defines a $W$-valued quadratic form on $\bm{L}$.
\item
The $W$-quadratic space $\bm{L}$ is self-dual of rank $6$, and 
\[
\bm{L} = \{ x\in \bm{L}_\Q  :  x \bm{D}\subset\bm{D} \}. 
\]
\item
If  $C(\bm{L})$ denotes the Clifford algebra of $\bm{L}$, the natural map
\[
C(\bm{L}) \to \End_W(\bm{D})
\]
induced by the inclusion $\bm{L} \subset \End_W(\bm{D})$ is an isomorphism.
Under this isomorphism, the even Clifford algebra is identified with the subalgebra of 
$\co_E$-linear endomorphisms in $\End_W(\bm{D})$.
 \end{enumerate}
\end{proposition}

\begin{proof}
Fix a basis of $\bm{D}_\Q$ as in Lemma \ref{lem:standard coords}, and suppose first that $\bm{\omega}$ is 
given by (\ref{standard omega}). An easy calculation shows that 
\begin{equation}\label{basis hodge}
\begin{array}{lr}
(e_1 \wedge e_2)^\star =  f_3\wedge f_4   &\quad   (f_3 \wedge f_4)^\star =  e_1\wedge e_2  \\
(e_1 \wedge e_3)^\star =  f_4 \wedge f_2  & \quad(f_4 \wedge f_2)^\star =  e_1\wedge e_3 \\
(e_1 \wedge e_4)^\star =  f_2\wedge f_3  &\quad (f_2 \wedge f_3)^\star =  e_1\wedge e_4  \\
(e_2 \wedge e_3)^\star =  f_1\wedge f_4  &\quad (f_1 \wedge f_4)^\star =  e_2\wedge e_3  \\
(e_2 \wedge e_4)^\star =  f_3\wedge f_1  &\quad (f_3 \wedge f_1)^\star =  e_2\wedge e_4 \\
(e_3 \wedge e_4)^\star =  f_1\wedge f_2  &\quad (f_1 \wedge f_2)^\star =  e_3\wedge e_4 
\end{array}
\end{equation}
 from which $(x^\star)^\star =x$ is obvious.  Now set $\bm{\omega}' = \alpha \bm{\omega}$ with $\alpha\in \co_E^\times$
 of norm $1$, and denote by $x\mapsto x^{\star\prime}$ the Hodge star operator defined by $\bm{\omega}'$.   It is 
 related to the Hodge star operator for $\bm{\omega}$ by $x^{\star\prime} = \alpha x^\star$, and hence
 \[
 (x^{\star\prime})^{\star\prime} = ( \alpha (\alpha x^\star) )^\star = \alpha \overline\alpha (x^\star)^\star = x.
 \]
 This proves the first claim in full generality.
 
 Keep $\bm{\omega}$ as in (\ref{standard omega}).
 For the second claim, one first checks that all $x,y \in \bigwedge^2_E \bm{D}$ satisfy the relation
 \begin{equation}\label{scalar relation}
 x \circ  y + y^\star  \circ x^\star =  - \langle x,y \rangle  
\end{equation}
in $\End_W(\bm{D})$.  Indeed, it suffices to prove this when $x$ and $y$ are pure tensors of the form
$e_i\wedge e_j$ and $f_i\wedge f_j$, and this can be done by brute force.  
Of course (\ref{scalar relation}) immediately implies  (\ref{quadratic calc})
for all $x\in \bm{L}$, proving the second claim for  $\bm{\omega}$.  The
validity of (\ref{scalar relation})  for any other $\bm{\omega}'$  follows by the reasoning
of the previous paragraph.

For the third claim, note that the quadratic form $Q(x)= - \langle x,x\rangle /2$ on $\bm{L}$ extends to a 
quadratic form on $\bigwedge^2_E\bm{D}$ by the same formula (using the standing hypothesis
that $p$ is odd), with associated bilinear form
\[
[x,y] = - \frac{ 1}{2} \cdot \mathrm{Tr}_{E/\Q_p} \langle x,y\rangle ,
\]
and that there is an orthogonal decomposition
\[
\bigwedge\nolimits^2_E\bm{D} = \bm{L} \oplus \{ x\in \bigwedge\nolimits^2_E\bm{D} : x^\star=-x \}.
\]
 The self-duality of $\bigwedge^2_E\bm{D}$ under $\langle \cdot , \cdot \rangle $ 
implies its self-duality under $[\cdot,\cdot]$, which then implies the self-duality of the orthogonal summand $\bm{L}$.  
The Hodge star operator acts on the  $W$-module 
\[
\bigwedge\nolimits^2_E\bm{D}  = \bigwedge\nolimits^2\epsilon_0\bm{D} \oplus \bigwedge\nolimits^2\epsilon_1\bm{D}
\]
of rank $12$  by interchanging the two summands on the right, and hence its submodule of fixed points, 
$\bm{L}$, has rank $6$.  Finally, set $\bm{L}'=  \{ x\in \bm{L}_\Q  :  x \bm{D}\subset\bm{D} \}$.
Certainly $\bm{L} \subset \bm{L}'$, and the 
quadratic form $Q(x) =x\circ x$ restricted to $\bm{L}'$ 
takes values in $W=W_\Q\cap \End_W(\bm{D})$.  Therefore
$
(\bm{L}')^\vee \subset \bm{L}^\vee = \bm{L} \subset \bm{L}' \subset (\bm{L}')^\vee,
$
and so equality holds throughout.

For the fourth claim, the self-duality of $\bm{L}$ implies that $\bm{L}/p\bm{L}$ is the unique  nondegenerate $k$-quadratic space
of dimension $6$, and so its Clifford algebra is isomorphic to $M_8(k)$.  This means that the induced map
\[
C(\bm{L}/p\bm{L}) \iso C(\bm{L})\otimes_W k   \to \End_W(\bm{D})\otimes_W k
\]
is a homomorphism between central simple $k$-algebras of the same dimension, and hence is an isomorphism.
It  now follows from Nakayama's lemma that $C(\bm{L}) \to \End_W(\bm{D})$ is an isomorphism.
Every $x\in \bm{L}$ satisfies $x\circ \bm{\iota}(\alpha) = \bm{\iota}(\overline{\alpha}) \circ x$, and hence
the composition of any two elements of $\bm{L}$ is $\co_E$-linear.  This implies that the even
Clifford algebra  is contained in $\End_{\co_E\otimes W}(\bm{D})$, and equality holds
because both are $W$-module direct summands of $C(\bm{L}) \iso \End_W(\bm{D})$ of the same rank.
\end{proof}

The operator  
\[
\Phi (a\wedge b)= p^{-1} (Fa)\wedge(Fb)
\] 
makes $\bigwedge^2_E \bm{D}_\Q$ into a slope $0$ isocrystal.  In terms of the inclusion 
$\bigwedge^2_E\bm{D}_\Q \subset \End_W(\bm{D})_\Q$, this operator is just
\[
\Phi (a\wedge b) = F \circ (a\wedge b) \circ F^{-1}.
\]
  As $\Phi$ commutes with the Hodge star operator, it stabilizes the subspace $\bm{L}_\Q$
and makes $\bm{L}_\Q$ into a slope $0$ isocrystal. In this way we obtain inclusions of $\Q_p$-vector spaces
 \begin{equation}\label{special endo}
\bm{L}_\Q^\Phi\subset (\bigwedge\nolimits^2_E \bm{D}_\Q)^\Phi \subset \End( \bm{G} )_\Q,
 \end{equation}
where the $\Phi$ superscripts denote the subspaces of $\Phi$-fixed vectors.  
Endow   $\bm{L}_\Q^\Phi$ with the quadratic  form   $ Q(x) = x\circ x $, and the associated bilinear form
\[
[x,y]=  x\circ y + y\circ x = - \frac{1}{2} \cdot \mathrm{Tr}_{E/\Q_p} \langle x,y\rangle .
\]

\begin{remark}
The $6$-dimensional $E$-vector space $(\bigwedge^2_E \bm{D}_\Q)^\Phi$  is characterized as the space of  all
 Rosati-fixed $x\in \End( \bm{G} )_\Q$ satisfying $x\circ \bm{\iota}(\alpha) = \bm{\iota}(\overline{\alpha}) \circ x$ 
 for all $\alpha\in E$.  On the other hand,  the $6$-dimensional $\Q_p$-vector space $\bm{L}_\Q^\Phi$ depends on the choice of $\bm{\omega}$, 
 and so does not have a similar interpretation in terms of $\bm{\lambda}$ and $\bm{\iota}$  alone.   
\end{remark}
 
 While the \emph{subspace} $\bm{L}_\Q^\Phi\subset  \End( \bm{G} )_\Q$ depends on the choice of $\bm{\omega}$,
the following proposition shows that its isomorphism class as a quadratic space does not.
Denote by $\mathbb{H}$ the hyperbolic $\Q_p$-quadratic space of dimension $2$.

\begin{proposition}\label{prop:hasse}
For any choice of $\bm{\omega}$, the quadratic space $\bm{L}_\Q^\Phi$  has  Hasse invariant $-1$ and determinant
$
\det(\bm{L}_\Q^\Phi)=-\Delta
$
for any nonsquare $\Delta\in \Z_p^\times$.   Furthermore,  the  special orthogonal
group $\SO(\bm{L}_\Q^\Phi)$  is quasi-split and splits over $\Q_{p^2}$, and the space $\bm{L}_\Q^\Phi$
with the rescaled quadratic form $p^{-1} Q$ is isomorphic to $\mathbb{H}^2\oplus \Q_{p^2}$, where $\Q_{p^2}$
is endowed with its norm form $x\mapsto  \mathrm{Norm}_{\Q_{p^2}/\Q_p}(x)$.
\end{proposition}

\begin{proof}
First suppose that $\bm{\omega}$ is defined by (\ref{standard omega}).  In this case the relations
(\ref{basis hodge}) show that the vectors
\begin{align*}
x_1 &= e_1\wedge e_2 + f_3\wedge f_4 &
x_2 &= e_3\wedge e_4 + f_1\wedge f_2 \\
x_3 &= e_1\wedge e_3 + f_4\wedge f_2 &
x_4 &= e_4\wedge e_2 + f_1\wedge f_3 \\
x_5 &= e_1\wedge e_4 + f_2\wedge f_3 &
x_6 &= e_2\wedge e_3 + f_1\wedge f_4 
\end{align*}
form a basis of $\bm{L}_\Q$.  In this basis the operator $\Phi$ takes the block diagonal form
\[
\Phi = \left(
\begin{matrix}
0& p \\
p^{-1}&0 \\
& & 0& 1\\
& & 1& 0\\
& & & & 0& 1\\
& & & & 1 & 0
\end{matrix} \right)\circ \sigma,
\]
and the matrix of $Q$ is
\[
\big( [ x_i,x_j] \big) = 
\left( \begin{matrix}
0& -1 \\
-1&0 \\
& & 0& -1\\
& & -1& 0\\
& & & & 0& -1\\
& & & & -1 & 0
\end{matrix} \right) .
\]
Fix any nonsquare $\Delta\in \Z_p^\times$ and let  $u\in W^\times$ be a square root of $\Delta$.  The vectors
\begin{align*}
y_1 &= p x_1 + x_2 &
y_2 &= u(px_1 -   x_2 ) \\
y_3 & = x_3+x_4  &
y_4 &= u(x_3 -  x_4) \\
y_5 & = x_5 + x_6 &
y_6 & = u(x_5 - x_6)
\end{align*}
form an orthogonal basis of $\bm{L}_\Q^\Phi$ with
\begin{equation}\label{Q matrix}
\left( \frac{[ y_i,y_j]}{2}  \right) = 
\left( \begin{matrix}
-p \\
&   p\Delta \\
& & -1 \\
& & & \Delta \\
& & & & - 1\\
& & & & & \Delta
\end{matrix} \right) ,
\end{equation}
from which one easily computes the determinant $-\Delta$ and Hasse invariant  $(-p,p\Delta) = -1$
of $\bm{L}_\Q^\Phi$.   As a nondegenerate quadratic space over $\Q_p$ is determined by its rank, 
determinant, and Hasse invariant, the remaining claims are 
easily checked for this special choice of $\bm{\omega}$.

Now suppose  $\bm{\omega}'=\alpha\bm{\omega}$ for some  $\alpha\in \co_E^\times$ of 
norm $1$.  Hilbert's Theorem 90 implies that there is some $\eta\in \co_E^\times$ satisfying
$\eta \overline{\eta}^{-1} = \alpha$.  Denote by  $x\mapsto x^{\star\prime}$ the Hodge star operator 
defined by $\bm{\omega}'$, by $\bm{L}' \subset \bigwedge_E^2\bm{D}$ the submodule of Hodge fixed vectors,
and by $Q'$ the quadratic form $x\circ x$ on $\bm{L}'$.
Using the relation $x^{\star\prime}=\alpha x^\star$, it is easy to see that the function $x\mapsto \eta x$ defines an 
isomorphism of quadratic spaces
\[
(\bm{L}_\Q^\Phi, \eta\overline{\eta} Q) \iso (\bm{L}_\Q^{\prime\Phi} , Q').
\]
In particular, there is a basis of $\bm{L}_\Q^{\prime\Phi}$ such that the quadratic form $Q'$
is given by $\eta\overline{\eta}$ times the matrix of (\ref{Q matrix}).  The Hasse invariant 
and determinant (modulo squares) of the matrix in (\ref{Q matrix}) are unchanged if the matrix is multiplied by any element of $\Z_p^\times$,
and so $\bm{L}_\Q^{\prime\Phi}$ has the same determinant and Hasse invariant as $\bm{L}_\Q^\Phi$. 
\end{proof}

 From now on we fix, once and for all, any  $\bm{\omega}$ as in Lemma \ref{lem:omega}.


\subsection{An exceptional isomorphism}


Define the unitary similitude group
\[
\GU( \bm{D}_\Q)  = \left\{ g\in  \Aut_{E\otimes W}(\bm{D}_\Q) : g^*\bm{\lambda} =\nu(g)\bm{\lambda} \mbox{ for some } \nu(g)\in W_\Q^\times \right\},
\]  
and set 
\[
\GU^0( \bm{D}_\Q ) =\left\{ g\in \GU( \bm{D}_\Q) : \nu(g)^2=\det(g) \right\}.
\]
The action $\action$ of  $\GU( \bm{D}_\Q) $  on $\End_W(\bm{D})_\Q$ defined by
$
g\action x = g\circ x\circ g^{-1}
$
 leaves invariant the subspace $\bigwedge^2_E\bm{D}_\Q$, and satisfies
\begin{equation}\label{wedge action}
g\action (a\wedge b) = \nu(g)^{-1}\cdot (ga)\wedge(gb).
\end{equation}
Using this formula one checks that the action of the subgroup 
$\GU^0( \bm{D}_\Q)$ commutes with the  Hodge star operator on $\bigwedge^2_E\bm{D}_\Q$, 
and  so preserves the subspace $\bm{L}_\Q$.

The \emph{canonical involution} $x\mapsto x'$ on the Clifford algebra
$C(\bm{L})$ is the unique $W$-linear endomorphism satisfying 
$
(x_1\cdots x_k)' = x_k\cdots x_1
$
for all $x_1,\ldots, x_k\in \bm{L}$, and the  spinor similitude group of $\bm{L}_\Q$ is
\[
\GSpin(\bm{L}_\Q )= \big\{ g\in C_0(\bm{L})_\Q^\times : g \bm{L}_\Q g^{-1} = \bm{L}_\Q \mbox{ and } g'g\in W_\Q^\times \big\}.
\]
Here $C_0(\bm{L})$ is the even Clifford algebra. From  \cite{Bass} or \cite{Shimura} we have the  exact sequence
\[
1 \to W_\Q^\times \to \GSpin(\bm{L}_\Q )   \to  \SO(\bm{L}_\Q) \to 1.
\]

\begin{proposition}\label{prop:gspin}
There is an isomorphism
\begin{equation}\label{exceptional}
 \GSpin(\bm{L}_\Q) \iso \GU^0( \bm{D}_\Q) 
\end{equation}
compatible with the action of both groups on $\bm{L}_\Q$.
In particular, the action of $\GU^0( \bm{D}_\Q)$ on $\bm{L}_\Q$ determines  an exact sequence
\[
1 \to W_\Q^\times \to \GU^0( \bm{D}_\Q)  \map{g\mapsto g\action} \SO(\bm{L}_\Q) \to 1.
\]
\end{proposition}

\begin{proof}
By Proposition \ref{prop:basic hodge} the inclusion of $\bm{L}$ into $\End_W(\bm{D})$
induces an isomorphism $C(\bm{L}) \iso \End_W(\bm{D})$, under which 
$C_0(\bm{L})\iso \End_{\co_{E}\otimes W}(\bm{D}).$
We will prove  that the induced isomorphism 
\[
C_0(\bm{L} ) _\Q^\times \iso \Aut_{E\otimes W}(\bm{D}_\Q)
\]
 restricts to an isomorphism (\ref{exceptional}).
Note that every element $x\in \bm{L}$, viewed as an endomorphism of $\bm{D}$, satisfies
$
\langle x a,b \rangle = - \overline{ \langle a ,xb\rangle}
$
(indeed, this already holds for every $x\in \bigwedge_E^2\bm{D}$).
Thus $\langle g a, b\rangle = \langle a , g' b\rangle$
for every $g\in C_0(\bm{L})$ and $a,b\in \bm{D}$.

One inclusion of (\ref{exceptional}) is obvious: if $g\in \GU^0( \bm{D}_\Q)$ then, as noted above, the conjugation action of $g$ on 
$C(\bm{L})_\Q \iso \End_W(\bm{D})_\Q$ preserves the subspace $\bm{L}_\Q$.  The relation
$\langle g a, gb \rangle = \langle a, g'g b\rangle$
implies that $\nu(g)=g'g$, and so $g\in  \GSpin(\bm{L}_\Q)$.

For the other inclusion, start with a $g\in \GSpin(\bm{L}_\Q )$.  The relation
$\langle ga,gb \rangle =  (g'g) \langle a,b\rangle$
shows that $g\in \GU(\bm{D}_\Q)$.  To show that $\nu(g)^2=\det(g)$, fix any 
$x\in \bm{L}$ and any $y\in \bigwedge^2_E\bm{D}$ for which $\langle y, x \rangle\not=0$.  As $g\action x =gxg^{-1}$ lies in $\bm{L}_\Q$ by assumption,
the Hodge star operator fixes $g\action x$.  Thus
\[
(g \action y) \wedge (g\action x) = \langle g\action y , g\action x\rangle \bm{\omega} = \langle y,x\rangle \bm{\omega},
\]
where the second equality follows from (\ref{wedge action}).
On the other hand, the Hodge star operator fixes $x$, and so
\[
(g \action y) \wedge (g\action x) = \nu(g)^{-2} \det(g) (y\wedge x)= \nu(g)^{-2} \det(g)\langle y,x\rangle \bm{\omega}.
\]
This proves that $g\in \GU^0(\bm{D}_\Q)$, and completes the proof of (\ref{exceptional}).
\end{proof}

The similitude character $\nu : \GU^0(\bm{D}_\Q ) \to W_\Q^\times$ restricts to $x\mapsto x^2$ on the subgroup 
$W_\Q^\times$, and so descends to the \emph{spinor norm}
\[
\tilde{\nu} : \SO(\bm{L}_\Q)  \to W_\Q^\times / (W_\Q^\times)^2.
\]

\begin{remark}\label{excRemark}
The group $J$ defined in Section \ref{ss:RZ} is characterized by
\[
J = \{ g\in \GU(\bm{D}_\Q) : g\circ F = F\circ g \},
\]
and we define a subgroup
\[
J ^0 = \{ g\in \GU^0(\bm{D}_\Q) : g\circ F = F\circ g \}.
\]
The isomorphism (\ref{exceptional}) restricts to an isomorphism
$\GSpin(\bm{L}_\Q^\Phi) \iso J^0$, and hence there is an exact sequence
\[
1 \to \Q_p^\times \to J^0 \to \SO(\bm{L}_\Q^\Phi) \to 1,
\]
which identifies  $J^\mathrm{der}\iso \Spin(\bm{L}_\Q^\Phi)$.
See \cite[Proposition  IV.15.27]{BookInvolutions} for similar exceptional isomorphisms.
\end{remark}


\subsection{Dieudonn\'e lattices and special lattices}
\label{ss:dieudonne and special}


In this subsection we show that the  $k$-points of $\mathscr{N}$ can be identified with the
set of homothety classes of certain lattices in $\bm{D}_\Q$, which we call \emph{Dieudonn\'e lattices}.  
We then use the inclusion
\[
\bm{L}_\Q \subset \End_W(\bm{D}_\Q).
\]
to construct  a bijection  between the set of homothety classes of Dieudonn\'e lattices and a set
of \emph{special lattices} in the slope $0$ isocrystal  $\bm{L}_\Q$.  Thus the points of 
$\mathscr{N}(k)$ are parametrized by these special lattices.

In fact, the proof of Theorem \ref{thm:comparison} below requires that we establish such a bijection 
not just over $k$, but over any extension field $k'\supset k$.
Let $W'$ be the Cohen ring of $k'$.  Thus $W'$ is the unique, up to isomorphism, complete discrete
valuation ring of mixed characteristic with residue field $W'/pW' \iso k'$.  The inclusion $k\to k'$
induces an injective ring homomorphism $W\to W'$, and we set $\bm{D}'=\bm{D}\otimes_W W'$
and $\bm{L}'=\bm{L}\otimes_W W'$.  There is a unique continuous ring homomorphism $\sigma:W'\to W'$
reducing to the Frobenius on $k'$, and the $\sigma$-semi-linear operators $F$ and $\Phi$ on $\bm{D}_\Q$
and $\bm{L}_\Q$ have unique $\sigma$-semi-linear extensions to $\bm{D}'_\Q$ and $\bm{L}'_\Q$.
Similarly the symplectic  and Hermitian forms  on $\bm{D}_\Q$ and the  quadratic form 
on $\bm{L}_\Q$ have natural extensions to $\bm{D}'_\Q$ and $\bm{L}'_\Q$.

Note that the operators $F$ and $\Phi$ are surjective on $\bm{D}_\Q$ and $\bm{L}_\Q$, respectively,
but this need not be true of their extensions to $\bm{D}'_\Q$ and $\bm{L}'_\Q$.
If $D\subset \bm{D}_\Q'$ is a $W'$-submodule then so is its preimage $F^{-1}(D)$, but its image 
$F(D)$ need not be.   Denote by $F_*(D)$ the $W'$-submodule
 generated by $F(D)$.  Similarly,  denote by  $\Phi_*(L)$  the $W'$-submodule generated by $\Phi(L)$ 
for a $W'$-submodule $L\subset\bm{L}_\Q'$.

For any $W'$-lattice $D\subset\bm{D}_\Q'$, set $D_1=F^{-1}(pD)$.

\begin{definition}\label{def:special lattice}
A \emph{Dieudonn\'e lattice} in $\bm{D}'_\Q$ is an $\co_E$-stable $W'$-lattice $D\subset \bm{D}'_\Q$  such that 
\begin{enumerate}
\item
$pD \subset D_1 \subset D$,
\item
$D^\vee= c  D$ for some $c\in \Q_p^\times$,
\item
$D=F_*( F^{-1}(D))$.
\end{enumerate}
Here the superscript $\vee$ denotes dual lattice with respect to the symplectic form $\bm{\lambda}$,
or, equivalently, with respect to the Hermitian form $\langle\cdot,\cdot\rangle$.
\end{definition}

The \emph{volume} of a lattice $D\subset \bm{D}_\Q'$ is the $W'$-submodule
\[
\mathrm{Vol}(D) = \bigwedge\nolimits^8 D \subset \bigwedge\nolimits^8\bm{D}_\Q',
\]
By considering the slopes of the isocrystal $\bm{D}_\Q$  one can show that 
$\mathrm{Vol}(F_*(D)) = p^4 \cdot \mathrm{Vol}(D)$.  However, taking preimages of
lattices may change volumes in  unexpected ways:  a lattice $D\subset \bm{D}'_\Q$ satisfies 
\[
\mathrm{Vol}(F^{-1}(D)) \subset p^{-4} \cdot \mathrm{Vol}(D),
\]
but equality holds if and only if $F_*(F^{-1}(D))=D$.   In particular,  the condition 
$D=F_*(F^{-1}(D))$ in Definition \ref{def:special lattice} is equivalent to 
$\mathrm{Vol}(D_1)=p^4 \mathrm{Vol}(D)$, and so one could replace  
 (3) in the definition of Dieudonn\'e lattice by 
\begin{enumerate}
\item[(3${}^\prime$)]  $\mathrm{dim}_{k'}(D_1/pD)=4$.
\end{enumerate}
The volume of a lattice in $\bm{L}_\Q'$ is defined in the analogous way, but now  
$\mathrm{Vol}(\Phi_*(L)) = \mathrm{Vol}(L)$ for any  lattice $L\subset\bm{L}_\Q'$.

\begin{proposition}\label{prop:dieudonne signature}
Suppose $D$ is a Dieudonn\'e lattice.
The $\co_E$-stable $k'$-subspace $D_1/pD \subset D/pD$ is  Lagrangian 
with respect to the nondegenerate symplectic form $c \bm{\lambda}$, and every $\alpha\in \co_E$ 
acts on $D/D_1$ with characteristic polynomial 
\begin{equation}\label{dieudonne signature}
\det(T- \iota(\alpha) ; D/D_1 ) = (T-\psi_0(\alpha))^2(T-\psi_1(\alpha))^2.
\end{equation}
\end{proposition}

\begin{proof}
For any $a,b\in D_1$ we have
\[
c\bm{\lambda}(a,b)^\sigma = p^{-1} c \bm{\lambda} (Fa,Fb) \in p \bm{\lambda}(cD,D) = pW.
\]
This shows that $D_1/pD$ is isotropic.  It is maximal isotropic, as $D_1/pD$ has dimension $4$.
 Lemma \ref{lem:standard coords} implies that 
\[
\bigwedge\nolimits^4 F_*(\epsilon_0 M ) = p^2 \cdot \bigwedge\nolimits^4 \epsilon_1 M,
\]
as submodules of $\bigwedge\nolimits^4 \epsilon_1 \bm{D}_\Q'$, for any lattice $M\subset \bm{D}_\Q'$.
Applying this with $M=D_1$ shows that  $\epsilon_1 D/\epsilon_1D_1$ has dimension $2$. 
The same argument shows that
$\epsilon_0 D/\epsilon_0D_1$ has dimension $2$, and  (\ref{dieudonne signature}) follows.
\end{proof}

\begin{corollary}\label{cor:dieudonne bijection}
There is a bijection $\mathscr{M}(k') \iso  \{ \mbox{Dieudonn\'e lattices in }\bm{D}_\Q' \}.$
\end{corollary}

\begin{proof}
If $k=k'$ then this is immediate from the equivalence of categories between Dieudonn\'e modules and
$p$-divisible groups: to any point $(G,\iota,\lambda,\varrho) \in \mathscr{M}(k)$ we let $D$
be the Dieudonn\'e module of $G$, viewed as a lattice in $\bm{D}_\Q$ using the isomorphism of 
isocrystals $\varrho:D_\Q \iso \bm{D}_\Q$.  For general $k'$ the argument is the same, using Zink's theory
of windows \cite{Zink} in place of Dieudonn\'e modules.
\end{proof}

\begin{theorem}\label{thm:the bijection}
Given a  Dieudonn\'e lattice $D$, set
\[
L =\{ x\in \bm{L}'_\Q : x D_1\subset D_1 \} 
\quad \mbox{ and } \quad
L^\sharp  =\{ x\in \bm{L}'_\Q : x D\subset D \}.
\]
The rule $D\mapsto (L, L^\sharp )$ defines a bijection from 
$p^\Z\backslash \{ \mbox{Dieudonn\'e lattices in $\bm{D}'_\Q$} \}$
to the set of all pairs of self-dual lattices $(L , L^\sharp)$ in $\bm{L}'_\Q$ such that 
\begin{enumerate}
\item
$\Phi_*(L) = L^\sharp$,
\item
$(L+L^\sharp)/L$ has length $1$.
\end{enumerate}
Moreover,  $L + L^\sharp = \{ x\in \bm{L}_\Q' : x D_1 \subset D \}.$
\end{theorem}

The proof of Theorem \ref{thm:the bijection} will be given in the next subsection.

\begin{definition}
A \emph{special lattice} is a  self-dual $W'$-lattice $L\subset \bm{L}'_\Q$  such that 
 \[
 \length \big( (L+\Phi_*(L) ) / L \big)=1.
 \]
\end{definition}

Obviously any pair of self-dual lattices $(L,L^\sharp)$ appearing in Theorem \ref{thm:the bijection}  
is determined by its first element, and in fact the function $L\mapsto (L,\Phi_*(L))$ establishes
a bijection between the set of special lattices and the set of pairs of self-dual  lattices $(L , L^\sharp)$
such that   $\Phi_*(L) = L^\sharp$ and  $(L+L^\sharp)/L$ has length $1$.
The only thing to check is   the self-duality of
$\Phi_*(L)$ for a special lattice $L$.  The inclusion $\Phi_*(L) \subset \Phi_*(L)^\vee$
is clear from the self-duality of $L$ and the relation   $[\Phi x,\Phi y] = [x,y]^\sigma$.  Equality holds because
$\mathrm{Vol}(\Phi_*(L)) = \mathrm{Vol}(L)$ and $L$ is self-dual.
The following corollary is now simply a restatement of Theorem \ref{thm:the bijection}.

\begin{corollary}\label{cor:the bijection}
The rule $D\mapsto \{ x\in \bm{L}'_\Q : x D_1\subset D_1 \} $ defines a bijection 
\[
p^\Z\backslash \{ \mbox{Dieudonn\'e lattices in $\bm{D}'_\Q$} \} \iso \{ \mbox{special lattices in }\bm{L}_\Q' \}.
\]
\end{corollary}


\subsection{Proof of Theorem \ref{thm:the bijection}}


In this subsection we prove Theorem \ref{thm:the bijection}.
Say that a $W'$-lattice $D\subset \bm{D}'_\Q$ is \emph{nearly self-dual} if $D^\vee=cD$ for some 
$c\in \Q_p^\times$.

\begin{lemma}\label{lem:bijection lemma 1}
The construction $D\mapsto \{x\in \bm{L}_\Q' : xD \subset D\}$  establishes a bijection 
\[
p^\Z\backslash \{ \mbox{nearly self-dual lattices $D\subset \bm{D}'_\Q$} \} \iso 
\{ \mbox{self-dual lattices  $L^\sharp\subset \bm{L}'_\Q$} \}.
\]
\end{lemma}

\begin{proof}
Start with a nearly self-dual lattice $D$, and set $L^\sharp=\{x\in \bm{L}_\Q' : xD \subset D\}$.
The condition $D^\vee = c D$ implies that there is some $g\in \GU^0( \bm{D}'_\Q )$
such that $D=g\bm{D}'$, and hence $L^\sharp=g\action \bm{L}'$.  As $g\action$ respects the 
quadratic form $Q$, the self-duality of $\bm{L}'$ implies the self-duality of $L^\sharp$.  Conversely, if we start with a 
self-dual $L^\sharp\subset \bm{L}'_\Q$, the Clifford algebra $C(L^\sharp)$ is a maximal order
in $C(\bm{L}'_\Q) \iso \End_{W'}(\bm{D}'_\Q)$, and so there is, up to scaling, a unique lattice $D\subset \bm{D}'_\Q$ satisfying 
\begin{equation}\label{integral clifford}
C(L^\sharp) = \End_W( D ).
\end{equation}
Choose any $h\in \SO(\bm{L}'_\Q)$ such that $L^\sharp=h\bm{L}'$, and lift $h$ to an element 
$g\in \GU^0(\bm{D}'_\Q)$.  By rescaling $g$  we may arrange to have $D=g \bm{D}'$, and the self-duality of 
$\bm{D}'$ implies  $D^\vee=\nu(g)^{-1} D$.  
\end{proof}

\begin{lemma}\label{lem:bijection lemma 3}
Suppose $D\subset\bm{D}_\Q'$ is nearly self-dual, $L^\sharp\subset\bm{L}_\Q'$ is self-dual, and 
$L^\sharp$ and $D$ are related by 
$L^\sharp = \{x\in \bm{L}_\Q' : xD \subset D\}$.
If $x\in L^\sharp/pL^\sharp$ is any nonzero isotropic vector, viewed as an endomorphism of $D/pD$ using (\ref{integral clifford}),
the kernel of $x$  is an $\co_E$-stable Lagrangian subspace with respect to $c\bm{\lambda}$.  Conversely, 
if $\mathscr{D}_1 \subset D/pD$ is an $\co_E$-stable Lagrangian subspace then 
$\{ x\in L^\sharp/pL^\sharp : x\mathscr{D}_1=0\}$ is an isotropic line in $L^\sharp/pL^\sharp$.  This construction establishes a bijection 
\[
\{ \mbox{isotropic lines in $L^\sharp/pL^\sharp$} \} \iso \{ \mbox{$\co_E$-stable Lagrangian subspaces in $D/pD$} \}.
\]
If $\mathscr{L}_1 \subset L^\sharp / pL^\sharp$ corresponds to $\mathscr{D}_1 \subset D/pD$ under this bijection, then
\begin{equation}\label{perp lemma}
\mathscr{L}_1^\perp = \{ x\in L^\sharp/pL^\sharp : x \cdot \mathscr{D}_1 \subset \mathscr{D}_1 \}.
\end{equation}
\end{lemma}

\begin{proof}
Abbreviate $\mathscr{L}=L^\sharp/pL^\sharp$ and $\mathscr{D}=D/pD$, so that $\mathscr{D}$ is the unique simple left
module over the Clifford algebra $C(\mathscr{L}) \iso M_8(k')$.  In particular 
$C(\mathscr{L}) \iso \mathscr{D}^8$ as left $C(\mathscr{L})$-modules.  If $x\in \mathscr{L}$ is 
any nonzero isotropic vector,  the kernel and image of left multiplication by $x$ on 
$C(\mathscr{L})$ are equal, and hence the kernel and image of 
$x\in \End(\mathscr{D})$ are also equal.  In particular $\ker(x)$ has dimension $4$.  
The relation $\alpha x=x\overline{\alpha}$ 
for all $\alpha\in \co_E$ shows that $\ker(x)$
is $\co_E$-stable, and the relation $(c\bm{\lambda}) ( x s,t) = (c\bm{\lambda})( s,xt)$ 
implies that $\ker(x)=x\mathscr{D}$ is  totally isotropic.

If $x,y\in \mathscr{L}$  are  nonzero isotropic vectors with $\ker(x)=\ker(y)$ then, from the discussion above, 
$\ker(x)=y\mathscr{D}$ and $\ker(y)=x\mathscr{D}$.  In particular $[x,y]=x\circ y+y\circ x=0$.  If 
$x$ and $y$ are not colinear then (after possibly extending scalars) 
we can find a   $z\in \mathscr{L}$ such that $k'x+k'y+k'z$ is a maximal
isotropic subspace of $\mathscr{L}$.  The left ideal $C(\mathscr{L})xyz$ has dimension eight as a $k$-vector space,
and so we must have $\mathscr{D}\iso C(\mathscr{L})xyz$ as left $C(\mathscr{L})$-modules.  But it is easy to see 
by direct calculation that 
the kernels of left multiplication by $x$ and $y$ on $C(\mathscr{L})xyz$ are different.  This contradiction shows that 
$x$ and $y$ are colinear, and so $x\mapsto \ker(x)$ establishes an injection $\mathscr{L}_1\mapsto \mathscr{D}_1$
from the set of isotropic lines in $\mathscr{L}$
to the set of $\co_E$-stable Lagrangian subspaces in $\mathscr{D}$.  

Endow $\mathscr{D}$ with the $\co_E\otimes_{\Z_p} k'$-valued Hermitian form induced by $c\langle\cdot,\cdot\rangle$.
Exactly as in Proposition \ref{prop:gspin} there is an isomorphism of $k'$-groups
$
\GSpin(\mathscr{L}) \iso \GU^0(\mathscr{D}).
$
This isomorphism is compatible, in the obvious sense, with the map $\mathscr{L}_1\mapsto \mathscr{D}_1$,
and so the image of the map is stable under the action of $\GU^0(\mathscr{D})$.  But $\GU^0(\mathscr{D})$
acts transitively on the set of $\co_E$-stable Lagrangian subspaces in $\mathscr{D}$, proving surjectivity.

Finally, we verify (\ref{perp lemma}).
If $\mathscr{L}_1$ corresponds to $\mathscr{D}_1$ under our bijection, then 
$\mathscr{D}_1= \ker(y)=y\mathscr{D}$ for any nonzero  $y\in \mathscr{L}_1$,
and an elementary argument (using $k'\cap yC(\mathscr{L})=0$ for the middle $\Longleftarrow$)
shows that 
\[
x\perp y \iff  xy+yx=0\iff xy C(\mathscr{L}) \subset yC(\mathscr{L})
\iff  xy\mathscr{D} \subset y\mathscr{D}_1.
\]
\end{proof}

\begin{proof}[Proof of Theorem \ref{thm:the bijection}]   
Suppose first that $D$ is a Dieudonn\'e lattice.  Using the relations $D=F_*(F^{-1}(D))$ and 
$\langle Fv,Fw\rangle = p\langle v,w\rangle^\sigma$, one can show that the near self-duality of $D$
implies that $D_1=F^{-1}(pD)$ is also nearly self-dual.   Lemma \ref{lem:bijection lemma 1} then implies that the lattices
\begin{align}
L & = \{x\in \bm{L}_\Q' : x D_1\subset D_1\} \label{stabilizers} \\
L^\sharp & =\{x\in \bm{L}_\Q' : x D\subset D\} \nonumber
\end{align}
are self-dual.   The relation $\Phi(x) \circ F = F\circ x$ for all $x\in \bm{L}_\Q'$ implies that
$\Phi_*(L) \subset L^\sharp$, and equality must hold as
\[
\mathrm{Vol}( \Phi_*(L) ) = \mathrm{Vol}(L) = \mathrm{Vol}(L^\sharp).
\]
By Proposition \ref{prop:dieudonne signature}   the $k$-subspace $D_1/pD \subset D/pD$ is  $\co_E$-stable 
and Lagrangian, and so it follows from Lemma \ref{lem:bijection lemma 3}  that
\[
\mathscr{L}_1 = \{  x\in L^\sharp/pL^\sharp : x (D_1/pD) =0 \}
\]
is an isotropic line in $L^\sharp/pL^\sharp$ with orthogonal complement
\[
\mathscr{L}_1^\perp = \{ x\in L^\sharp/pL^\sharp : x (D_1/pD) \subset D_1/pD \}.
\]
On the other hand, $L \cap L^\sharp=\{ x\in L^\sharp : xD_1 \subset D_1\}$, and so
\[
(L+ L^\sharp)/L \iso L^\sharp/(L\cap L^\sharp) \iso  (L^\sharp/pL^\sharp)/\mathscr{L}_1^\perp
\]
has length $1$.

Now suppose we start with a pair of self-dual lattices $(L,L^\sharp)$ such that
$\Phi_*(L)=L^\sharp$ and $(L+L^\sharp)/L$ has length $1$.  By  Lemma
\ref{lem:bijection lemma 1} there are unique (up to scaling) nearly self-dual lattices $D_1$ and $D$ in $\bm{D}_\Q'$
satisfying  (\ref{stabilizers}).  Set $L_0=L\cap L^\sharp$, so that $L^\sharp/L_0$ has length $1$,
and pick any nonzero $y\in L^\sharp/L_0$.  The Clifford algebra $C(L^\sharp)$ satisfies
\[
C(L^\sharp)=C(L_0) + y C(L_0),
\]
where $C(L_0) \subset C(L^\sharp)$ is the $W$-subalgebra generated by $L_0$, and so
\[
C( L^\sharp )D_1 = C(L_0)D_1 + y C(L_0)D_1 = D_1+y D_1.
\]
This implies
$
C(L^\sharp)pD_1 \subset D_1 \subset C(L^\sharp) D_1.
$
The self-duality of $L^\sharp$ implies that $C(L^\sharp)$ is a maximal order in 
$
C(\bm{L}_\Q') = \End_{W'}(\bm{D}_\Q'),
$ 
and so we must have $C(L^\sharp)=\End_{W'}(D)$.  As the lattice $C(L^\sharp)D_1$ is obviously stabilized by $C(L^\sharp)$, it must have the form
$C(L^\sharp)D_1=p^k D$ for some integer $k$.    Thus after rescaling $D_1$ we may assume that $C(L^\sharp)D_1=D$ and 
 \[
 p D\subset D_1 \subset D.
 \]   
 The relation $\Phi(x) \circ F = F\circ x$  implies
 \[
 C(L^\sharp) F_*(D_1) =C(\Phi_*(L))  F_*(D_1) =
 F_* ( C(L)D_1) =F_*(D_1),
 \]
and so $F_*(D_1)=p^k D$ for some $k$.  Combining
 \[
p^{8k} \cdot \mathrm{Vol}(D) = \mathrm{Vol}(F_*(D_1)) = p^4\cdot  \mathrm{Vol}(D_1)
 \]
and 
\[
p^8 \cdot \mathrm{Vol}(D) \subset \mathrm{Vol}(D_1) \subset \mathrm{Vol}(D)
\]
shows that in fact  $F_*(D_1)=p D$.  The relations $D_1=F^{-1}(pD)$ and $F_*(F^{-1}(D))=D$
follow easily from this,  proving that $D$ is a Dieudonn\'e lattice.

It only remains to prove  that $L+L^\sharp = \{ x\in \bm{L}_\Q' : x D_1 \subset D \}.$
The inclusion $L+L^\sharp \subset \{ x\in \bm{L}_\Q' : x D_1 \subset D \}$ is obvious from (\ref{stabilizers}). 
On the other hand, each side contains $L^\sharp$ with quotient having length $1$
(for the right hand side this follows from Proposition \ref{prop:dieudonne signature} and  Lemma \ref{lem:bijection lemma 3}).  Thus equality holds.
 \end{proof}


\subsection{Vertex lattices and the Bruhat-Tits stratification}
\label{ss:vertex}


If we start with a $k$-point $(G,\lambda,i,\varrho) \in \mathscr{M}(k)$ and let $D$ be the covariant 
Dieudonn\'e module of $G$, then  $\varrho(D) \subset \bm{D}_\Q$
is a Dieudonn\'e lattice.  This construction is simply the $k'=k$ case of the bijection
\[
\mathscr{M}(k) \iso \{ \mbox{Dieudonn\'e lattices in $\bm{D}_\Q$} \} 
\]
of Corollary \ref{cor:dieudonne bijection}.
Combining this with Corollary \ref{cor:the bijection} yields a bijection
\begin{equation}\label{pointwise bijection}
\mathscr{N}(k)  \iso   \{ \mbox{special lattices in $\bm{L}_\Q$} \}
\end{equation}
defined by
\[
(G,\lambda,i,\varrho) \mapsto
 \{ x\in \bm{L}_\Q :  x  \varrho(D_1) \subset \varrho(D_1)  \},
 \]
 where $D_1=VD$. Moreover,  Theorem \ref{thm:the bijection} implies 
 that the special lattice  
 \[
 L= \{ x\in \bm{L}_\Q :  x  \varrho(D_1) \subset \varrho(D_1)  \}
 \] 
 satisfies
 \[
 \Phi(L) =  \{ x\in \bm{L}_\Q :  x  \varrho(D) \subset \varrho(D)  \}.
 \]

The next step is to show that the special lattices come in natural families, indexed by certain \emph{vertex lattices}
in the $\Q_p$-quadratic space $\bm{L}_\Q^\Phi$.  Using this and the bijection (\ref{pointwise bijection}), we will then
express the reduced scheme underlying $\mathscr{N}$ as a union of closed subvarieties indexed by vertex lattices.

\begin{definition}
A \emph{vertex lattice} is a $\Z_p$-lattice $\Lambda \subset \bm{L}_\Q^\Phi$  such that 
\[
p \Lambda \subset\Lambda^\vee \subset \Lambda .
\]
The \emph{type} of $\Lambda$ is  $t_\Lambda = \mathrm{dim}_k ( \Lambda / \Lambda^\vee ).$
\end{definition}

\begin{lemma}
The type of a vertex lattice is either $2$, $4$, or $6$.
\end{lemma}

\begin{proof}
Let $\Lambda$ be a vertex lattice.  Proposition \ref{prop:hasse} implies that $\ord_p( \det( \Lambda ))$ is even, 
from which it follows that the type of $\Lambda$ is also even.  If $\Lambda$ has type $0$ then $\Lambda$ is 
self-dual, and hence admits a  basis such that the matrix of $Q$ is diagonal with diagonal entries in $\Z_p^\times$.
But this implies that  $\bm{L}_\Q^\Phi$ has Hasse invariant  $1$, contradicting Proposition \ref{prop:hasse}.
\end{proof}

The proof of the following proposition is identical to that  of  Proposition 4.1 of \cite{RTW}.  
See also Lemma 2.1 of \cite{Vol}.

\begin{proposition}\label{prop:zink}
Let $L\subset \bm{L}_\Q$ be a special lattice, and define 
\[
L^{(r)}= L + \Phi(L) + \cdots + \Phi^r(L).
\]  
There is an integer $d\in \{ 1,2,3\}$ such that 
\[
L=L^{(0)} \subsetneq L^{(1)} \subsetneq \cdots\subsetneq L^{(d)} =L^{(d+1)}.
\]
For each $L^{(r)} \subsetneq L^{(r+1)}$ with $0\le r<d$ the quotient $L^{(r+1)}/L^{(r)}$ 
is annihilated by $p$, and satisfies $\mathrm{dim}_k(L^{(r+1)}/L^{(r)}) = 1$.  Moreover, 
\[
\Lambda_L = \{ x\in L^{(d)}  : \Phi(x) = x\}
\]
is a vertex lattice of type $2d$, and satisfies  
$
\Lambda_L^\vee =\{ x\in L : \Phi(x) = x \}.
$  
\end{proposition}

By (\ref{special endo}) each vertex lattice $\Lambda$ determines a collection of quasi-endomorphisms 
$
\Lambda^\vee \subset \End(\bm{G})_\Q.
$  
 Define a closed formal subscheme
$
\widetilde{\mathscr{M}}_\Lambda \subset \mathscr{M}
$
as the locus of points $(G,\iota, \lambda,\varrho)$ such that 
\[
\varrho^{-1} \Lambda^\vee \varrho  =  \{  \varrho^{-1} \circ x \circ \varrho : x\in \Lambda^\vee  \} \subset \End(G) .
\]
In other words, the locus where the   quasi-endomorphisms $\varrho^{-1} \Lambda^\vee \varrho$  of $G$
are actually integral.  Set
$
\widetilde{\mathscr{N}}_\Lambda = p^\Z \backslash \widetilde{\mathscr{M}}_\Lambda,
$
and let $\mathscr{N}_\Lambda$ be the reduced $k$-scheme underlying $\widetilde{\mathscr{N}}_\Lambda$.
The  bijection (\ref{pointwise bijection}) identifies
\begin{align}\label{special bijection}
\mathscr{N}_\Lambda(k)  
&=  \{ \mbox{special lattices  $L$ such that $\Lambda^\vee \subset \Phi(L)$} \} \\
&=  \{ \mbox{special lattices  $L$ such that $\Lambda^\vee \subset L$} \}  \nonumber \\
&=  \{ \mbox{special lattices  $L$ such that $\Lambda_L \subset \Lambda$} \}. \nonumber
\end{align}
The same proof used in \cite[Proposition 4.3]{RTW} shows that
\[
\mathscr{N}_{\Lambda_1} \cap \mathscr{N}_{\Lambda_2} = 
\begin{cases}
\mathscr{N}_{ \Lambda_1 \cap \Lambda_2} & \mbox{if $\Lambda_1\cap \Lambda_2$ is a vertex lattice} \\
\emptyset & \mbox{otherwise}
\end{cases}
\]
where the left hand side is understood to mean the reduced subscheme underlying the scheme-theoretic 
intersection.

\begin{proposition}\label{prop:proj}
Each $k$-scheme $\mathscr{N}_\Lambda$ is projective.
\end{proposition}

\begin{proof}
Let $R_\Lambda$ be the $W$-subalgebra of $\End_W(\bm{D}_\Q)$ generated by $\Lambda^\vee$,
and let $\tilde{R}_\Lambda$ be a maximal order in $\End_W(\bm{D}_\Q)$ containing $R_\Lambda$. 
It follows from the isomorphism $C(\bm{L}_\Q) \iso \End_W(\bm{D}_\Q)$ of Proposition \ref{prop:basic hodge}
that $R_\Lambda$ is a $W$-lattice in $\End_W(\bm{D}_\Q)$, and hence $\tilde{R}_\Lambda/R_\Lambda$
is killed by some power of $p$, say $p^M$.  Up to scaling by powers of $p$, there is a unique 
$W$-lattice $\tilde{D}\subset \bm{D}_\Q$ such that $\tilde{R}_\Lambda \tilde{D}=\tilde{D}$.

Now suppose $(G,  \iota, \lambda, \varrho)$ is a $k$-point of $\mathscr{N}_\Lambda$.  The quasi-isogeny
$\varrho$ determines (up to scaling)  a $W$-lattice $D\subset \bm{D}_\Q$ satisfying
$R_\Lambda D =D$.   It follows from $\tilde{R}_\Lambda D=\tilde{D}$ that 
\[
 p^{M} \tilde{D} \subset D\subset  \tilde{D},
\]
after possibly rescaling $D$, and so there are integers $a<b$, independent of the point
$(G, \iota, \lambda, \varrho)$, such that $p^a \bm{D} \subset D \subset p^b\bm{D}$.
It  follows from this bound and \cite[Corollary 2.29]{RZ} that $\mathscr{N}_\Lambda$ is a closed
subscheme of a projective scheme, hence is projective.
\end{proof}

An obvious corollary of Proposition \ref{prop:zink}  is that  every special lattice $L$ contains  
some $\Lambda^\vee$ (take $\Lambda=\Lambda_L$), and hence
\[
\mathscr{N}_\mathrm{red}  = \bigcup_\Lambda \mathscr{N}_\Lambda ,
\]
where the subscript $\mathrm{red}$ indicates the underlying reduced scheme.
This union is not disjoint, as $\mathscr{N}_{\Lambda_1}^\pm \subset \mathscr{N}_{\Lambda_2}^\pm$
whenever $\Lambda_1 \subset \Lambda_2$.  Define
\[
\mathscr{N}_\Lambda^\circ =
 \mathscr{N}_\Lambda \smallsetminus \bigcup_{\Lambda'\subsetneq \Lambda} \mathscr{N}_{\Lambda'},
\]
so that  (\ref{pointwise bijection}) identifies
\[
\mathscr{N}^\circ_\Lambda(k) =  \{ \mbox{special lattices  $L$ such that $\Lambda_L=\Lambda$ } \}.
\]
It follows easily that
$
\mathscr{N}_\Lambda  = \biguplus_{\Lambda'\subset \Lambda} \mathscr{N}^{\circ}_{\Lambda'} ,
$
and that
\begin{equation}\label{bruhat-tits strat}
\mathscr{N}_\mathrm{red} = \biguplus_\Lambda \mathscr{N}^{\circ}_\Lambda .
\end{equation}
 Abbreviate  
$
\mathscr{N}_\Lambda^\pm = \mathscr{N}_\Lambda \cap \mathscr{N}^\pm
$
and  $\mathscr{N}_\Lambda^{ \pm \circ } = \mathscr{N}^\circ_\Lambda \cap \mathscr{N}^\pm$.
By analogy with \cite{RTW}, \cite{Vol}, and \cite{VolWed}, we call the decomposition 
(\ref{bruhat-tits strat}) the \emph{Bruhat-Tits stratification} of $\mathscr{N}_\mathrm{red}$.
This terminology should be taken with a grain of salt: unlike in \emph{loc.~cit.~}the strata in
(\ref{bruhat-tits strat}) are not in bijection with the vertices in the Bruhat-Tits building of the 
group $J^\mathrm{der}$.  See Sections \ref{ss:building} and \ref{ss:building2} below.

\begin{remark}\label{rem:vertex action}
One could also define an \emph{$E$-vertex lattice} to be an $\co_E$-lattice 
$ \Lambda_E \subset  (\bigwedge_E^2 \bm{D}_\Q )^\Phi$  such that 
$p \Lambda_E \subset\Lambda_E^\vee \subset \Lambda_E ,$
where the dual lattice is taken with respect to $\langle \cdot ,\cdot \rangle$.  The rule
$\Lambda \mapsto \co_E \cdot \Lambda$ establishes a bijection between vertex lattices and $E$-vertex lattices,
with inverse $\Lambda_E\mapsto \{ x\in \Lambda_E : x^\star =x \} . $
The action (\ref{wedge action}) of $\GU(\bm{D}_\Q)$ on $\bigwedge_E^2\bm{D}_\Q$ restricts to an action of 
$J$ on  $(\bigwedge_E^2 \bm{D}_\Q )^\Phi$, and induces an action of $J$ on the set of all $E$-vertex lattices.
In particular, $J$ acts on the set of all vertex lattices.  This action is compatible with the action of $J$ on $\mathscr{N}$
defined in Section \ref{ss:RZ}, in the obvious sense:   
$g \mathscr{N}_\Lambda = \mathscr{N}_{g\action \Lambda}$.  The restriction of this action 
to the subgroup $J^0$ of Remark \ref{excRemark} factors through the surjection 
$J^0 \to \SO(\bm{L}_\Q^\Phi)$, and agrees with the obvious action of  $\SO(\bm{L}_\Q^\Phi)$ 
on the set of vertex lattices.
\end{remark}


\subsection{The Bruhat-Tits building}
\label{ss:building}


In \cite[20.3]{bookGarrett} one finds a description of the Bruhat-Tits building of  $\SO(\bm{L}_\Q^\Phi)$
in terms of lattices.  See also \cite[1.16]{TitsCorvallis}. We will translate this description into the 
 language of our vertex lattices. Consider the set $\VV^\mathrm{adm}$ of all vertex lattices 
 $\Lambda$ of type $2$ or $6$.  We call such vertex lattices {\sl admissible}, and define an adjacency 
 relation $\sim$ in $\VV^\mathrm{adm}$  as follows: distinct admissible 
 vertex lattices are adjacent ($\Lambda\sim \Lambda'$)  if either 
\begin{enumerate}
\item 
$\Lambda'\subset\Lambda$ or $\Lambda'\subset \Lambda$, 
\item
or, $\Lambda$ and $\Lambda'$ are both type $6$ and 
\[
\dim_{\F_p}(\Lambda/\Lambda\cap \Lambda')=\dim_{\F_p}(\Lambda'/\Lambda\cap \Lambda')=1,
\]
\[
\dim_{\F_p}(\Lambda+\Lambda'/\Lambda )=\dim_{\F_p}(\Lambda+\Lambda'/\Lambda')=1.
\]
\end{enumerate}
If $\Lambda$ and $\Lambda'$ are of type $6$ and are adjacent, then $\Lambda\cap\Lambda'$ 
is  a vertex lattice of type $4$ (so  is not admissible).
We construct an abstract simplicial  complex with set of vertices $\VV^\mathrm{adm}$ as 
follows: An $m$-simplex ($0\le m\le 2$) of $\VV^\mathrm{adm}$ is a subset  of $m+1$ 
admissible vertex lattices $\Lambda_0,\Lambda_1, \ldots, \Lambda_m$ 
which are mutually adjacent.    The group $\SO(\bm{L}_\Q^\Phi)$  acts simplicially on $ \VV^\mathrm{adm}$  by 
$g\in \SO(\bm{L}_\Q^\Phi)$ taking $\Lambda$ to $g\cdot\Lambda$.

Now consider the Bruhat-Tits building $\mathcal{BT}$ of $\SO(\bm{L}_\Q^\Phi)$.
We will use the same symbol $\mathcal{BT}$ to denote the underlying simplicial complex.

\begin{proposition}
There is an $\SO(\bm{L}_\Q^\Phi)$-equivariant simplicial bijection $\mathcal{BT}\iso\VV^\mathrm{adm}$.
Furthermore, every vertex lattice of type $4$ is  contained in precisely two vertex lattices of type $6$, and is equal to their intersection.
\end{proposition}

\begin{proof}  
Define a new quadratic space $(V_0,Q_0) = (\bm{L}_\Q^\Phi, p^{-1}Q)$, and note that, by
Proposition \ref{prop:hasse}, $V_0\iso  \mathbb{H}^{2}\oplus \Q_{p^2}$.  The rule $\Lambda \mapsto p\Lambda$
defines a bijection from the set of vertex lattices in $\bm{L}_\Q^\Phi$ to the set of lattices $L\subset V_0$ satisfying
\[
L\subset L^*\subset p^{-1}L.
\]
Here $L^*$ is the dual lattice of $L$ with respect to the quadratic form $Q_0$.
The isomorphism $\mathcal{BT}\iso\VV^\mathrm{adm}$ now follows from the description 
and properties of the affine building of   $\SO(\bm{L}_\Q^\Phi)\iso\SO(V_0)$ found in \cite[20.3]{bookGarrett}. 

If $\Lambda$ is a type $4$ vertex lattice, the latttice $L=p\Lambda$ in $V_0$ satisfies 
$\dim(L^*/L) =2$.  Moreover, the $k$-quadratic space $L^*/L$ is a hyperbolic plane (choose a basis of $L$
for which the bilinear form has diagonal matrix, and use the fact that $V_0$ has Hasse invariant $1$),
and so contains exactly two isotropic lines.  Those lines have the form $L_1/L$ and $L_2/L$,
and  $p^{-1}L_1$ and $p^{-1}L_2$ are the unique type $6$ vertex lattices containing $\Lambda$.
\end{proof}

We can also construct a simplicial complex $\VV$ with  vertices the set of {\sl all} 
vertex lattices as follows (compare to  \cite[\S 3]{RTW}). We call two distinct vertex lattices $\Lambda$ and $\Lambda'$
neighbors  if $\Lambda\subset \Lambda'$ or $\Lambda'\subset \Lambda$. 
An $m$-simplex ($m\le 2$) in $\VV$ is formed by vertex lattices
$\Lambda_0, \Lambda_1,\ldots , \Lambda_m$ such that any two members of this set are neighbors. 
The vertex lattices of type $4$ are in bijection with pairs of adjacent type $6$ vertex lattices. 
Hence a vertex lattice of type $4$ corresponds to an edge in the Bruhat-Tits building between 
type $6$ vertex lattices.  From basic properties of the Bruhat-Tits building, we deduce the following.

\begin{corollary}\label{cor:connected}
The group $\SO(\bm{L}_\Q^\Phi)$ acts transitively on the set of vertex lattices  of a given type,
and any two vertex lattices are connected by a sequence of adjacent vertices in $\VV$.  
In particular, the group $J$, and even the subgroup $J^0$, 
acts transitively on the set of vertex lattices  of a given type (under the action of 
Remark \ref{rem:vertex action}).
\end{corollary}


\section{Deligne-Lusztig varieties and the Bruhat-Tits strata}
\label{strata}


 In this section we  show that,  for any vertex lattice $\Lambda$, the varieties
\[
\mathscr{N}_\Lambda=\mathscr{N}_\Lambda^+ \uplus \mathscr{N}_\Lambda^-
\quad \mbox{and} \quad 
\mathscr{N}_\Lambda^\circ=\mathscr{N}_\Lambda^{\circ+} \uplus \mathscr{N}_\Lambda^{\circ -}
\]
of Section \ref{ss:vertex} can be identified with varieties over $k$ defined purely in terms of the 
linear algebra of the $k$-quadratic space  $\Omega= (\Lambda/\Lambda^\vee)\otimes_{\F_p}k$.


\subsection{Deligne-Lusztig varieties}

Let us recall the general definition of Deligne-Lusztig varieties. Suppose that $G_0$ is a connected reductive group over the finite field $\F_p$,
and set  $G=G_0\otimes_{\F_p}k$.  We will also use the symbol $G$ to denote the abstract group of $k$-valued points
of $G_0$. Denote by $\Phi:G\to G$ the Frobenius moprhism.   By Lang's theorem $G_0$ is quasi-split, and so we may choose 
  a maximal torus $T \subset G$ and  a Borel subgroup containing $T$,   both defined over $\F_p$.
The Weyl group $W$ that corresponds to the pair $(T, B)$ is acted upon by $\Phi$, and  the group $W$ with its $\Phi$-action
does not depend on our choices.  In fact,  in \cite{DeligneLusztig} a Weyl group $W$ with $\Phi$-action is defined as a 
projective limit over all choices of pairs $(T, B)$, without having to assume that these pairs are $\Phi$-stable.

Let $\Delta^*=\{\alpha_1,\ldots , \alpha_n\}$ be the set of simple roots corresponding to the pair $(T, B)$,
and consider the corresponding simple reflections $s_i=s_{\alpha_i}$ in the Weyl group $W$. 
For $I\subset \Delta^*$, let $W_I$ be the subgroup of $W$ generated by $\{ s_i : i\in I\}$, and consider the corresponding 
parabolic subgroup $P_I=BW_IB$. The quotient $G/P_I$ parametrizes parabolic subgroups of $G$ of type $I$.  Suppose $J\subset \Delta^*$ is another subset with corresponding standard parabolic $P_J$. Since 
\[
G=\biguplus_{w\in W_I  \backslash W/W_J} P_IwP_J
\]
we have a bijection 
\[
P_I\backslash G/P_J \iso W_I\backslash W/W_J
\]
Composing this with $G/P_I\times G/P_J\to P_I\backslash G/P_J$ given by $(g_1, g_2)\mapsto g_1^{-1}g_2$ 
defines the \emph{relative position invariant}
\[
\inv: G/P_I\times G/P_J\to W_I\backslash W/W_J.
\]
The Frobenius $\Phi: G\to G$ induces $\Phi: G/P_I\to G/P_{\Phi(I)}$.
\begin{definition}  
For $w\in W_I\backslash W/W_{\Phi(I)}$, the \emph{Deligne-Lusztig variety} $X_{P_I}(w)$ is the locally closed
reduced subscheme of $G/P_I$ with $k$-points  
\[
X_{P_I}(w) = \{  gP_I \in G/P_I :  \inv(g, \Phi(g))=w\}.
\]
\end{definition}
The   variety  $X_{P_I}(w)$ is actually defined over the unique extension of degree $r$ of $\F_p$ in $k$, 
where $r$ is the smallest positive integer for which $\Phi^r(I)=I$.

\begin{proposition}\label{prop:DL properties}
The  Deligne-Lusztig variety $X_{P_I}(w)$  is smooth of pure dimension 
\[
\dim X_{P_I}(w) = \ell(w)+\dim (G/P_{I\cap \Phi(I)})-\dim (G/P_I).
\]
If  $I=\Phi(I)$ then $\dim X_{P_I}(w) = \ell_I(w) - \ell(w_I)$,
where $w_I$ is the longest element in $W_I$, and $\ell_I(w)$ is the maximal length of an element in $W_I w W_I$.
Taking $I=\emptyset$, the variety $X_B(w)$ is irreducible of dimension $\dim X_B(w)=\ell(w)$.
\end{proposition}

\begin{proof}
This is standard. See  \cite[Section 3.4]{VolWed}, for example.
\end{proof}

\begin{remark}\label{rem:identity} 
If $w=1$, then $X_{P_I}(1)$ can be identified with the intersection of the image of 
the closed immersion $G/P_{I\cap \Phi(I)}\hookrightarrow G/P_I\times G/P_{\Phi(I)}$
with the graph of Frobenius $\Phi: G/P_I\to G/P_{\Phi(I)}$. In particular  $X_{P_I}(1)$
is projective. If \[\bigcup_{r\geq 0} \Phi^r(I)=\Delta^*\]  then $X_{P_I}(1)$ is  also irreducible, by \cite{BonnafeRouquier}.
\end{remark}


\subsection{An even orthogonal group} 
\label{BTstrata}


We now consider the case that $G_0$ is a nonsplit special orthogonal group 
in an even number of variables.  Let $\Omega_0$ be an $\F_p$-vector space of dimension $2d$  equipped with a 
nondegenerate nonsplit quadratic form.  There is a basis $\{e_1,\ldots , e_d, f_1, \ldots, f_d\}$ of 
$
\Omega=\Omega_0\otimes_{\F_p} k
$ 
such that   $\langle e_1, \ldots, e_d\rangle $ and $\langle f_1,\ldots,  f_d\rangle$ are isotropic,  
$[e_i, f_j]=\delta_{ij}$,  and the Frobenius 
\[
\Phi = {\rm id}\otimes \sigma
\] 
acting on  $\Omega$ fixes $e_i$ and $f_i$ for $1\leq i\leq d-1$, and interchanges $e_d$ with $f_d$. 
Note that $\Omega$ contains no $\Phi$-invariant Lagrangian subspaces.
Abbreviate $G_0=\SO(\Omega_0)$ and $G=\SO(\Omega)$.

Denote by $\mathrm{OGr}(r)$ the scheme whose functor of points assigns to a $k$-scheme 
$S$ the set of all   totally isotropic local $\co_S$-module direct summands
$
\LL \subset \Omega\otimes_k\co_S
$
of rank $r$.   In particular, $\mathrm{OGr}(d)$ is the moduli space of Lagrangian subspaces of $\Omega$.
Denote by $\mathrm{OGr}(d-1,d)$  the scheme whose functor of points assigns to a $k$-scheme 
$S$ the set of all flags of totally isotropic local $\co_S$-module direct summands
$
\LL_{d-1} \subset \LL_d  \subset \Omega\otimes_k\co_S
$
of rank $d-1$ and $d$, respectively.  The following lemma is elementary.

\begin{lemma}\label{easyperp}
For each totally isotropic  local $\co_S$-module direct summand \[\LL_{d-1} \subset \Omega\otimes_k\co_S\] of rank $d-1$,  
there are exactly two totally isotropic local $\co_S$-module direct summands
of rank $d$ containing $\LL_{d-1} $.  
\end{lemma}

In other words, the forgetful map 
$
\mathrm{OGr}(d-1,d) \to \mathrm{OGr}(d-1)
$
is a two-to-one cover.  In fact, the Grassmannian $\mathrm{OGr}(d-1,d)$ has two connected components,
which  are interchanged by the action of any orthogonal transformation of determinant $-1$.  Each of the 
two components maps isomorphically to $\mathrm{OGr}(d-1)$ under the forgetful map.    
 Label
the two components as 
\[
\mathrm{OGr}(d-1,d) = \mathrm{OGr}^+(d-1,d) \uplus  \mathrm{OGr}^-(d-1,d)
\]
in such a way that the flags
\begin{equation}\label{flag +}
\langle e_1, \ldots, e_{d-1}\rangle \subset \langle e_1, \ldots , e_{d-1},e_d\rangle  
\end{equation}
and
\begin{equation}\label{flag -}
\langle e_1, \ldots, e_{d-1}\rangle   \subset \langle e_1, \ldots , e_{d-1}, f_d \rangle 
\end{equation}
define  $k$-points of  $\mathrm{OGr}^+(d-1,d)$ and $\mathrm{OGr}^-(d-1,d)$, respectively.

Denote by $\mathscr{X}\subset \mathrm{OGr}(d)$ the reduced closed subscheme  with $k$-points 
\[
\mathscr{X} = \{  \LL  \in \mathrm{OGr}(d) :  \dim( \LL + \Phi( \LL) ) =d+1 \}.
\]
There is a closed immersion $\mathscr{X} \to \mathrm{OGr}(d-1,d)$  sending 
\[
\LL \mapsto  \LL \cap \Phi(\LL) \subset \LL,
\]
and the open and closed subvariety 
$
\mathscr{X}^\pm = \mathscr{X} \cap \mathrm{OGr}^\pm(d-1,d)
$
of $\mathscr{X}$ is identified with
\begin{equation}\label{twisted grass}
\mathscr{X}^\pm = \{ \LL_{d-1} \subset \LL_d \in \mathrm{OGr}^\pm(d-1,d) :  \LL_{d-1} \subset  \Phi(\LL_d) \}.
\end{equation}

 \begin{remark}\label{rem:swap}
 Although we have defined $\mathrm{OGr}(d-1,d)$ and $\mathscr{X}$ as $k$-schemes, they both have natural 
 $\F_p$-structures.  The Frobenius morphism from  $\mathrm{OGr}(d-1,d)$ to itself   
 interchanges the flags (\ref{flag +}) and (\ref{flag -}), and  hence interchanges the two connected components.  It 
 follows that  $\mathscr{X}^\pm \iso \sigma^*\mathscr{X}^\mp$, and that
 the individual components $\mathscr{X}^+$ and  $\mathscr{X}^-$ have natural $\F_{p^2}$-structures.
 \end{remark}

Our choice of basis of $\Omega$ determines a maximal $\Phi$-stable torus $T \subset G$. Set 
\begin{align}\label{std flag}
\FF_i^\pm & = \langle e_1, \ldots , e_i \rangle \mbox{ for } 1\leq i\leq d-1 \\
\FF_d^+  & = \langle e_1, \ldots, e_{d-1}, e_d \rangle \nonumber \\
\FF_d^-  & = \langle e_1, \ldots, e_{d-1}, f_d \rangle  \nonumber .
\end{align} 
This gives two ``standard" isotropic flags  $\FF^+_\bullet$ and $\FF^-_\bullet$ in $\Omega$ satisfying
$
\FF^\pm_\bullet=\Phi(\FF^\mp_\bullet) .
$ 
These flags have the same stabilizer  $B\subset G$, which is a $\Phi$-stable Borel subgroup containing $T$.
The corresponding set of simple reflections in the Weyl group is 
\[
\{s_1, \ldots, s_{d-2}, t^+, t^-\}
\]
where
\begin{itemize}
\item $s_i$  interchanges $e_i$ with $e_{i+1}$,  $f_i$ with $f_{i+1}$, and fixes the other basis elements.

\item $t^+$ interchanges $e_{d-1}$ with $e_d$, $f_{d-1}$ with $f_d$, and fixes the other basis elements.

\item $t^-$ interchanges $e_{d-1}$ with $f_d$, $f_{d-1}$ with $e_d$, and fixes the other basis elements.
\end{itemize}
Notice that  $\Phi(s_i)=s_i$, and $\Phi(t^\pm)=t^{\mp}$, and so the products
\[
w^\pm=t^{\mp}s_{d-2}\cdots s_2s_1
\] 
 are  \emph{Coxeter elements}: products of exactly one representative from each $\Phi$-orbit 
in the set of simple reflections above. More generally, define $w_0=1$, $w^\pm_1=t^\mp$, and 
$$
w_r^\pm =t^\mp \cdot s_{d-2}\cdots s_{d-r}
$$
for $2\leq  r\leq d-1$.  In particular, $w_{d-1}^\pm=w^\pm$. 
Define parabolic subgroups 
\[
B=P_{d-1} \subset P_{d-2} \subset \ldots \subset P_0 \subset P^\pm
\]
of $G$  as follows:  set $ P_{d-1}=P_{d-2}=B$, and for $0\leq r\leq d-2$  let $P_r$ be the parabolic  corresponding to the set
  $\{s_1,\ldots , s_{d-(r+2)}\}$.  Define $P^\pm$ to be the maximal parabolic  corresponding to $\{s_1,\ldots , s_{d-2}, t^\pm\}$.
One can easily check that  $P_0$ is the stabilizer in $G$ of $\FF_{d-1}^\pm$, and so
\begin{equation}\label{flag 1}
 G/P_0 \iso \mathrm{OGr}(d-1) . 
\end{equation}
More generally,  $P_r$ is the stabilizer of the standard isotropic
flag $\FF^\pm_{d-r-1} \subset\cdots \subset \FF^\pm_{d}$
Similarly, $P^\pm$ is the stabilizer of the Lagrangian subspace $\FF_d^\pm$, and so 
\begin{equation}\label{flag 2}
G/P^\pm \iso \mathrm{OGr}(d).
\end{equation}

  \begin{proposition}
The isomorphism (\ref{flag 2}) identifies   $\mathscr{X}^\pm$with the Deligne-Lusztig variety 
$X_{P^\pm}(1)$.  In particular, $\mathscr{X}^\pm$  is    projective, irreducible, and smooth of dimension $d-1$.
\end{proposition}

\begin{proof}
Note that $P_0=P^+\cap P^-$, and that the Frobenius $\Phi$ interchanges $P^+$ and $P^-$. The 
two projections $G/P_0\to G/P^\pm$ combine to give  closed immersions
\begin{align*} 
i^+ & : G/P_0\to G/P^+\times G/P^- \\
 i^- & : G/P_0\to G/P^-\times G/P^+,
\end{align*}
while the Frobenius induces morphisms $\Phi: G/P^+\to G/P^-$ and 
$\Phi: G/P^-\to G/P^+$ with graphs $\Gamma^+_\Phi\subset G/P^+ \times G/P^-$
and $\Gamma^-_\Phi\subset G/P^-\times G/P^+$, respectively.
The isomorphisms (\ref{flag 1}) and (\ref{flag 2}) identify the  intersection of $\Gamma^\pm_\Phi$ and the image 
of $i^\pm$ with the set of flags $\LL_{d-1} \subset \LL_d \in \mathrm{OGr}^\pm(d-1,d)$ such that 
$\LL_{d-1} \subset \Phi(\LL_d)$.  By (\ref{twisted grass}), this intersection is isomorphic to   $\mathscr{X}^+$.
All of the claims now follow from Remark \ref{rem:identity}, together with the dimension formula
\[
\dim X_{P^\pm}(1) = \dim (G/P_0)-\dim (G/P^\pm) =d-1
\]
of Proposition \ref{prop:DL properties}.  
\end{proof}

It is also useful to view $\mathscr{X}^\pm$ as a closure of a Deligne-Lusztig variety in the flag variety $G/P_0$.

\begin{lemma}
The isomorphism $\mathrm{OGr}^\pm(d-1,d) \iso \mathrm{OGr}(d-1) \iso G/P_0$ identifies
$\mathscr{X}^\pm$ with the closure in $G/P_0$ of the Deligne-Lusztig variety
\[
X_{P_0}(t^\mp)=\{g\in G/P_0 :  \inv(g, \Phi(g))=t^\mp \} .
\]
\end{lemma}

\begin{proof}
Using (\ref{twisted grass}), we may characterize the $k$-points of $\mathscr{X}^\pm$ by
\[
\mathscr{X}^\pm = \{ \LL_{d-1} \subset \LL_d  \in \mathrm{Gr}^\pm(d-1 ,d ) : \Phi(\LL_{d-1})=\LL_{d-1} \mbox{ or } \LL_{d-1}+\Phi(\LL_{d-1})=\Phi(\LL_d) \}.
\]
Recalling the standard isotropic flags of (\ref{std flag}), the rule  $g \mapsto g \FF_{d-1}^\pm \subset g \FF_d^\pm$ 
defines an isomorphism
\[
X_{P_0}(1) \iso \{ \LL_{d-1} \subset \LL_d  \in  \mathrm{Gr}^\pm(d-1 ,d ) : \LL_{d-1} = \Phi( \LL_{d-1})  \},
\]
while the same rule defines an isomorphism
\[
X_{P_0}(t^\mp) \iso \{ \LL_{d-1} \subset \LL_d \in   \mathrm{Gr}^\pm(d-1 ,d ) : \LL_{d-1}+\Phi(\LL_{d-1} )= \Phi(\LL_d) \}.
\]
Thus $\mathscr{X}^\pm$ is the disjoint union of $X_{P_0}(1)$ and $X_{P_0}(t^\mp )$.  
Elementary properties of the Bruhat order (see Section 8.5 of \cite{Springer}, for example)  imply that 
\[
\overline{ X_{P_0}(t^\mp ) } = X_{P_0}(1)\uplus X_{P_0}(t^\mp ),
\]
completing the proof.
\end{proof}

The following  proof is essentially the same as \cite[Proposition 5.5]{RTW}.

\begin{proposition}\label{prop:stratification}
There is a stratification
\[
\mathscr{X}^\pm=\biguplus_{r=0}^{d-1} X_{P_r}(w^\pm_r)
\] 
of $\mathscr{X}^\pm$ into a disjoint union of  locally closed subvarieties.  The  stratum
$X_{P_r}(w^\pm_r)$ is smooth of pure dimension $r$,  and  has closure 
\[
\overline{ X_{P_r}(w^\pm_r) } = X_{P_0}(w^\pm_0) \uplus \cdots \uplus X_{P_r}(w^\pm_r).
 \] 
 The highest dimensional stratum $X_{P_{d-1}}(w^\pm_{d-1})=  X_B(w^\pm)$  is irreducible,  open, and dense.
\end{proposition}

\begin{proof}
 Suppose  $\LL$ is a  $k$-point of $\mathscr{X}^\pm\subset \mathrm{Gr}(d)$, and define
\begin{equation*}
\LL^{(i)}=\LL \cap \Phi(\LL)\cap \cdots \cap \Phi^i(\LL).
\end{equation*}
 An inductive argument using 
\[
\LL^{(i)}  \cap\Phi(\LL^{(i)} )=\LL^{ (i-1) }\cap\Phi(\LL^{( i-1) }) \cap \Phi^2(\LL^{ (i-1) }),
\]
shows that $\LL^{( i+1) }$ has codimension at most $1$ in $\LL^i$.
Denote by $\mathscr{X}^\pm_r \subset \mathscr{X}^\pm$ the reduced closed subscheme 
whose $k$-valued points are those $\LL$ satisfying $\LL^{ ( r+2) }=\LL^{(r+1)}$.
The complement $\mathscr{X}^\pm_{r} \smallsetminus \mathscr{X}^\pm_{r-1}$ is the locally closed subvariety of $\mathscr{X}^\pm$
consisting of those $\LL$ for which
\begin{equation}\label{flag}
\LL^{( r+2) }=\LL^{(r+1)}  \subsetneq \cdots \subsetneq \LL^{(1)}  \subsetneq \LL^{(0)} = \LL.
\end{equation}
Recalling that the parabolic subgroup  $P_r$ is the stabilizer of the standard isotropic
flag $\FF^\pm_{d-r-1} \subset\cdots \subset \FF^\pm_{d}$ of length $r+2$,
we obtain morphism $\mathscr{X}^\pm_{r} \smallsetminus \mathscr{X}^\pm_{r-1} \to G/P_r$ by sending 
$\LL$ to the flag (\ref{flag}).  By similar reasoning as  \cite[Proposition 5.5]{RTW}, this defines an isomorphism 
$\mathscr{X}^\pm_{r} \smallsetminus \mathscr{X}^\pm_{r-1} \iso X_{P_r}(w_r^\pm)$ with inverse
$g \mapsto g \FF_d^\pm.$  All claims now follow easily.
\end{proof}


\subsection{A few special cases}
\label{ss:DL special cases}

 
We continue to let $G_0=\SO(\Omega_0)$, where $\Omega_0$ is the nonsplit quadratic space
over $\F_p$ of dimension $2d$, $\Omega=\Omega_0\otimes k$, and $G=\SO(\Omega)$. 
In  the applications we will only need to consider $d\le 3$,  and in these cases the structure of the
$k$-variety $\mathscr{X}$ (with its $\F_{p^2}$-structure of Remark \ref{rem:swap}) can be made more explicit.

a) First suppose $d=1$.  In this case $\mathscr{X}^\pm$ is a single point, defined over $\F_{p^2}$.

b) Now suppose $d=2$.  In this case $P_1=P_0=B$,  and the stratification of Proposition \ref{prop:stratification} is
\[
\mathscr{X}^\pm=X_B(1)\uplus X_B(t^\mp)
\]
where $X_B(1)$ is a  $0$-dimensional closed subvariety, and the open stratum $X_B(t^\mp)$ 
has dimension $1$. The Dynkin diagram identity $D_2 = A_1\times A_1$   corresponds to an exceptional isomorphism
  $\Spin(\Omega) \iso \SL_2\times\SL_2$. 
Indeed, the even Clifford algebra $C_0(\Omega_0)$ is isomorphic to $M_{2}(\F_{p^2})$, and hence
$C_0(\Omega) \iso M_2(k) \times M_2(k)$.   This isomorphism restricts to an isomorphism of algebraic groups 
\[
\GSpin(\Omega) \iso \{ (x,y) \in \GL_2 \times \GL_2 : \det(x)=\det(y) \}
\]
over $k$, which in turn determines  an isomorphism of $k$-varieties
$
G/P_0 \iso   {\mathbb P}^1\times {\mathbb P}^1
$
in such a way that the Frobenius morphism on the left corresponds to $(x,y)\mapsto (\Phi(y), \Phi(x))$
on the right.  The subvarieties $\mathscr{X}^\pm \subset G/P_0$ are identified with 
\begin{align*}
\mathscr{X}^+ & =\{( \Phi(x) , x) :  x\in {\mathbb P}^1\}, \\
\mathscr{X}^- & =\{(x,  \Phi(x) ) :  x\in {\mathbb P}^1\}.
\end{align*}
Therefore, both $\mathscr{X}^+$ and $\mathscr{X}^-$ are isomorphic (over $\F_{p^2}$) to $ \mathbb{P}^1$. 
The closed stratum $X_B(1)$ corresponds to the $\F_{p^2}$-rational points of ${\mathbb P}^1$.

c) Finally, suppose $d=3$.  In this case 
\begin{equation}\label{strat3}
\mathscr{X}^\pm= X_{P_0 }(1)\uplus X_{B}(t^\mp)\uplus  X_{B}(t^\mp s_{1}).
\end{equation}
The open  stratum  $X_{B}(t^\mp s_{1})$ has dimension $2$, the stratum
 $X_{B}(t^\mp)$ is locally closed  of dimension $1$,   and the closed stratum $X_{P_0 }(1)$ has dimension $0$.
 To continue, we will use the Dynkin diagram isomorphism $D_3=A_3$, corresponding to an isomorphism  
 between the adjoint forms of $G$ and a unitary group in $4$ variables, as in Proposition \ref{prop:gspin}
and Remark \ref{excRemark}.   

Let  $\VV_0$ be  the  $4$-dimensional $\F_{p^2}$-vector space with basis $e_1,e_2,e_3,e_4$, endowed with the split 
$\F_{p^2}/\F_p$-Hermitian form defined by $\langle e_i, e_{5-j}\rangle=\delta_{ij}$. 
Denote by $U_0$ the unitary group of $\VV_0$, an algebraic group over $\F_p$, so that  
$U=U_0\times_{\F_p} k$ acts on $\VV=\VV_0\otimes_{\F_p} k$.
Recall the isomorphism $\F_{p^2}\otimes_{\F_p}  k \simeq k\oplus k$ 
defined by $x\otimes a\mapsto (xa , x^p a)$, and denote by $\epsilon_0$ and $\epsilon_1$  the idempotents 
that correspond to $(1, 0)$ and  $(0,1)$, so that 
\begin{equation}\label{finite splitting}
\epsilon_0\VV \iso \VV_0\otimes_{\F_{p^2}}k.
\end{equation}
The action of $U$ on $\epsilon_0\VV$ defines an isomorphism $U \iso \GL_4$.
The diagonal torus and the standard Borel subgroup of upper triangular matrices 
in $\GL_4$ give a maximal torus and a Borel subgroup of $U$,  both defined over $\F_p$.
Given  $0\leq r, s$ with $r+s=4$,  the pair   $(r, s)$, viewed as an ordered partition of $4$, 
defines a  parabolic subgroup $P_{(r,s)}$ containing $B$  with Levi component $\GL_r\times \GL_s$. 
The parabolic subgroup $P_{(r,s)}$ is  defined over
$\F_{p^2}$ and satisfies $\Phi(P_{(r,s)})=P_{(s, r)}$. 

Let $\mathrm{Gr}(r)$ be the Grassmanian of $r$-planes in $\epsilon_0 \VV$.
The above isomorphism $U \iso \GL_4$ induces an isomorphism $U/P_{(r,s)}\iso {\mathrm {Gr}}(r)$ 
defined over $\F_{p^2}$, and the Frobenius morphism $\Phi: U/P_{(r,s)}\to U/P_{(s, r)}$ corresponds to a 
morphism
\begin{equation}\label{unitary morphism}
\Phi: {\mathrm {Gr}}(r)\xrightarrow{ \ } {\mathrm {Gr}}(s)
\end{equation}
which can be described, as in  \cite{Vol},  as follows.  
Consider the $k$-valued form $\langle\langle \cdot ,\cdot  \rangle\rangle$ on 
(\ref{finite splitting}) defined  by
\[
\langle\langle x\otimes a, y\otimes b\rangle\rangle=\langle x, y\rangle \otimes ab^p.
\]
It is $k$-linear in the first variable but Frobenius semi-linear in the second.
  For a subspace $\mathcal U\subset  \epsilon_0 \VV$,
denote by  $\mathcal U^{\lefthalfcup}$   the perpendicular of $\UU$ for the form 
$\langle \langle \cdot ,\cdot \rangle\rangle$. If $\UU$ is $\F_{p^2}$-rational 
then $\UU^\lefthalfcup=\UU^\perp$. If $\dim_k(\UU)=r$  then $\dim_k(\UU^\lefthalfcup)=s$, and,
according to  \cite[Lemma 2.12]{Vol},  the morphism (\ref{unitary morphism})
is given   by $\Phi(\UU)=\UU^\lefthalfcup$ on $k$-valued points.
Using this description of $\Phi$,  the unitary Deligne-Lusztig variety 
$X_{P_{(r,s)} }(1) \subset \mathrm{Gr}(r)$ is seen to be
\[
X_{P_{(r,s)}}(1)= 
\begin{cases}
\{\UU\subset  \epsilon_0 \VV : \dim_k(\UU)=r, \ \UU\subset \UU^{\lefthalfcup}\}
& \mbox{ if } r\le s \\
\{\UU\subset  \epsilon_0 \VV : \dim_k(\UU)=r, \ \UU^\lefthalfcup\subset \UU \}
& \mbox{ if } s\le r . 
\end{cases}
\]
By Remark \ref{rem:identity} these Deligne-Lusztig varieties are projective and smooth.
 
By inspecting the Dynkin diagram identity $D_3=A_3$ we can see that the exceptional isomorphism between the
adjoint forms of $G$ and $U$ can be chosen so that
the parabolic subgroups $P^+$, $P^-$ and $P_0$ of $G$ correspond to $P_{(1,3)}$, $P_{(3,1)}$
and $P_{(1,3)}\cap P_{(3,1)}$ of $U \iso \GL_4$ respectively. Therefore, the Deligne-Lusztig variety $\mathscr{X}^{+}$
is isomorphic to the unitary  Deligne-Lusztig variety 
\[
X_{P_{(1,3)}}(1)=\{\UU\subset k^4:\dim_k(\UU)=1, \ \UU\subset \UU^\lefthalfcup\}.
\]
Similarly $\mathscr{X}^{-}$ is isomorphic to the unitary Deligne-Lusztig variety 
\[
X_{P_{(3,1)}}(1)=\{\UU\subset k^4:\dim_k(\UU)=3, \ \UU^\lefthalfcup\subset \UU\}.
\]
Therefore, both $\mathscr{X}^+$ and $\mathscr{X}^-$ are isomorphic over $\F_{p^2}$ to the 
smooth hypersurface in $\mathbb{P}^3$ given by the homogeneous equation:
\[
x_1x_4^p+x_2x^p_3+x_3x_2^p+x_4x_1^p=0.
\]
In fact, since all nondegenerate Hermitian forms on $\VV_0=\F_{p^2}^4$ are isomorphic
we can also determine equations for the unitary Deligne-Lusztig varieties using the Hermitian form 
given by the identity matrix $I_4$. This gives the Fermat hypersurface
\[
x_0^{p+1}+x_1^{p+1}+x_2^{p+1}+x_3^{p+1}=0,
\]
which is isomorphic to the surface above.

The stratification (\ref{strat3}) of $\mathscr{X}^+$ now corresponds to the stratification of the unitary
 Deligne-Lusztig variety $X_{P_{(1,3)}}(1)$  studied in \cite[Theorem 2.15]{Vol}. 
 The Frobenius $\sigma:k \to k$ defines an operator on $\VV$, which interchanges the two 
 summands $\VV=\epsilon_0\VV\oplus \epsilon_1\VV$.  
 Thus we obtain an operator  $\tau=\sigma^2$ on $\epsilon_0\VV$.
Any $k$-subspace $\UU\subset \epsilon_0\VV$ satisfies  
$\tau(\UU)=(\UU^\lefthalfcup)^\lefthalfcup$. The open $2$-dimensional
stratum of $X_{P_{(1,3)}}(1)$ has $k$-valued points corresponding to lines $\UU$ such that
\begin{align*}
\dim_k(\UU+\tau(\UU))&=2 \\ 
\dim_k(\UU+\tau(\UU)+\tau^2(\UU))&=3.
\end{align*}
The $1$-dimensional stratum has $k$-valued points corresponding to lines $\UU$ such that 
$\dim_k(\UU+\tau(\UU))=2$ and $\UU+\tau(\UU)$ is $\tau$-invariant (\emph{i.e.~}$\F_{p^2}$-rational). 
 Finally, the $0$-dimensional stratum consists 
of $k$-valued points corresponding to lines $\UU$ which are $\tau$-invariant.
In other words, the $0$-dimensional stratum of $\mathscr{X}^+$ is just the set of  $\F_{p^2}$-rational points.
For a $k$-valued point $\UU$ on the $1$-dimensional statum, set $\UU'=\UU+\tau(\UU)$.  This
 is an $\F_{p^2}$-rational plane with $\UU'^\lefthalfcup=\UU'^\perp=\UU'$.
The irreducible components of the $1$-dimensional stratum are parametrized
by such planes. Indeed,  also conversely, given an $\F_{p^2}$-rational plane $\UU'$ 
which is isotropic $\UU'=\UU'^\perp$, 
we obtain a closed subscheme of $X_{P_{(1,3)}}(1)$ with points corresponding
to all lines $\UU$ with $\UU\subset \UU'$.  This subscheme is isomorphic to $\mathbb{P}^1$ and gives the 
Zariski closure of the corresponding irreducible component  of the $1$-dimensional stratum. 

We can now also determine the number of  components of the strata:

$\bullet$ The $0$-dimensional stratum consists of $(p^3+1)(p^2+1)$ 
points. Indeed, observe 
that, as in \cite[Example 5.6]{VolWed}, we can calculate that the number of $\F_{p^2}$-valued points of the 
Fermat surface above is equal to $(p^3+1)(p^2+1)$. (Note that there is a typographical error in \emph{loc.~cit.}:
the summation for $\Sigma_l$ should start at $j=0$.)
This is equal to the number
of $\F_{p^2}$-rational lines $\UU\subset \F_{p^2}^4$ such that $\UU\subset \UU^{\perp}$
where the orthogonal is with respect to the (standard) 
Hermitian form $\langle\cdot,\cdot\rangle$ on $\F^4_{p^2}$. Therefore, we have $(p^3+1)(p^2+1)$
components of the $0$-dimensional stratum.

$\bullet$ The $1$-dimensional stratum has $(p^3+1)(p+1)$ 
 components. Note that by the above the Zariski closure of each component
is isomorphic to a projective line ${\mathbb{P}^1}$ over $\F_{p^2}$
and the corresponding component is 
 the complement of all $\F_{p^2}$-rational points
in this line. 
To determine the number of irreducible components of the $1$-dimensional stratum we start by counting the number
of such components 
whose closure passes a given $\F_{p^2}$-rational point,
i.e. the number of $\mathbb{P}^1$'s in our configuration that cross at that point:
By the above, this count is given by the number of 
$\F_{p^2}$-rational planes $\UU'$ which are isotropic $\UU'=\UU'^\perp$ 
and satisfy $\UU\subset \UU'\subset \UU^\perp$. These are given by 
$\F_{p^2}$-rational lines in $\UU^\perp/\UU$ which are isotropic for
the induced hermitian form. We can easily see that there are exactly $p+1$
such lines. Since there are a total of $(p^3+1)(p^2+1)$ $\F_{p^2}$-rational points
each belonging on $p+1$ projective lines which each have $p^2+1$ points we conclude that there are exactly $(p^3+1)(p+1)$ projective lines
in our configuration. 

 The same discussion  applies to $\mathscr{X}^-$.


\subsection{Deligne-Lusztig varieties and the Bruhat-Tits stratification} 


Now we relate the varieties $\mathscr{X}$ studied above to the  varieties
$\mathscr{N}_\Lambda^\circ \subset \mathscr{N}_\Lambda$ of  Section \ref{ss:vertex}.  
Fix a vertex lattice $\Lambda$ of type $2d\in \{2,4,6\}$, and endow the $2d$-dimensional 
$\F_p$-vector space
\[
\Omega_0=\Lambda/\Lambda^\vee
\]
with the nondegenerate $\F_p$-valued quadratic form $q(x)=pQ(x)$ induced by the quadratic form $Q$ on $\bm{L}_\Q^\Phi$.
Set $\Omega=\Omega_0\otimes_{\F_p} k$, and let $G$ be the special orthogonal group of $\Omega$.  
Note that $\Omega_0$ is nonsplit: the existence of a $d$-dimensional totally isotropic subspace in $\Omega_0$ 
would imply the existence of a self-dual lattice in $\bm{L}_\Q^\Phi$, contradicting the Hasse invariant calculation of 
Proposition \ref{prop:hasse}.

Recall from Section  \ref{BTstrata}   the reduced closed subscheme  $\mathscr{X}=\mathscr{X}_\Lambda\subset \mathrm{OGr}(d)$ 
whose $k$-points are the Lagrangian subspaces $\LL\subset \Omega$ with
\begin{equation}\label{special codim}
\dim_k (  \LL+\Phi(\LL)  ) =d+1.
\end{equation}
 The Lagrangian subspaces of  $\Omega$ are in bijection with the $W$-lattices $L\subset \bm{L}_\Q$ satisfying $L=L^\vee$ and
$\Lambda^\vee \subset L$, and the condition (\ref{special codim}) is equivalent to  $L$ being a  special lattice.
Combining this with (\ref{special bijection}), we obtain bijections
\[
\mathscr{X}_\Lambda(k) \iso \{ \mbox{special lattices  $L\subset \bm{L}_\Q$ such that $\Lambda^\vee \subset L$} \}
\iso \mathscr{N}_\Lambda(k).
\]

\begin{theorem}\label{thm:comparison}
There is an isomorphism of $k$-schemes $ \mathscr{N}_\Lambda\iso \mathscr{X}_\Lambda$
inducing the above bijection on $k$-points.  After possibly relabeling the two connected components of 
$\mathscr{X}_\Lambda =\mathscr{X}_\Lambda^+ \uplus \mathscr{X}_\Lambda^-$, we may assume that this isomorphism identifies
$ \mathscr{N}^\pm_\Lambda\iso \mathscr{X}_\Lambda^\pm$.
\end{theorem}

\begin{proof}
Let $R$ be  a reduced $k$-algebra of finite type.
 Given an $R$-valued point $(G, \iota, \lambda,\varrho)\in \mathscr{N}_\Lambda(R)$  there is an induced map of $\Z_p$-modules
\[
\Lambda^\vee \to \End(G)
\]
defined by $x\mapsto \varrho^{-1} \circ x \circ \varrho$.  Let $\mathcal{D}$ 
be the covariant Grothendieck-Messing crystal of 
$G$, evaluated at the trivial divided power thickening $\Spec(R)\to \Spec(R)$, 
so that $\mathcal{D}$ is a locally free $R$-module sitting in an exact sequence
\[
0 \to \mathcal{D}_1 \to \mathcal{D} \to \Lie(G) \to 0.
\]
The formation of the pair  $\mathcal{D}_1\subset \mathcal{D}$ is functorial in $G$, and so there
are  induced morphisms of $R$-modules
\[
 \psi: (\Lambda^\vee/p\Lambda^\vee) \otimes_{\F_p} R \to \End_R(\mathcal{D})
\]
and 
\[
 \psi_1: (\Lambda^\vee/p\Lambda^\vee) \otimes_{\F_p} R \to \End_R(\mathcal{D}_1)
 \]
 with $\ker(\psi)\subset \ker(\psi_1)$.
Given   $x\in \mathrm{ker}(\psi_1)$ and  $y\in (\Lambda^\vee / p\Lambda^\vee) \otimes_{\F_p} R$,  
the endomorphism
\[
[x , y] = x\circ y+y\circ x \in \End_R(\mathcal{D}_1)
\]
is  trivial.   Thus $\mathrm{ker}(\psi_1)$ is contained in the radical of the quadratic space 
$(\Lambda^\vee / p\Lambda^\vee) \otimes_{\F_p} R$, which is $(p\Lambda / p\Lambda^\vee) \otimes_{\F_p} R$. 
 Let  $\LL^\sharp\subset \mathcal{K}$
 be the images of $\ker(\psi) \subset \ker(\psi_1)$ under the obvious isomorphism
\[
(p\Lambda / p\Lambda^\vee) \otimes_{\F_p} R
\iso 
(\Lambda / \Lambda^\vee) \otimes_{\F_p} R \iso \Omega \otimes_k R.
\]

When $R=k$ the point $(G, \iota, \lambda ,\varrho)$ corresponds to some Dieudonn\'e lattice $D$
with $D_1=VD$,  and $\mathcal{D}_1 \subset \mathcal{D}$ is  canonically identified with  
$D_1/pD\subset  D/pD$.   Under these identifications
\begin{align*}
\mathrm{ker}(\psi) &= \{ x\in (p\Lambda/p\Lambda^\vee) \otimes_{\F_p} k : xD\subset pD \} \\
\mathrm{ker}(\psi_1) &= \{ x\in (p\Lambda/p\Lambda^\vee) \otimes_{\F_p} k : xD_1\subset pD \}, 
\end{align*}
and so 
\begin{align*}
\LL^\sharp  & = \{  x\in (\Lambda/\Lambda^\vee) \otimes k : xD\subset D \} \\
\mathcal{K}  & = \{  x\in (\Lambda/\Lambda^\vee) \otimes k : xD_1\subset D \}
\end{align*}
If we identify a subspace  of  $(\Lambda/\Lambda^\vee) \otimes k$ with the lattice in 
$\bm{L}_\Q$ that it generates, then  $\LL^\sharp$ corresponds to the lattice 
$L^\sharp = \{ x\in \bm{L}_\Q : x D\subset D\}$ of 
Theorem \ref{thm:the bijection}, and 
$\mathcal{K}$ corresponds to $L+L^\sharp=\{x\in \bm{L}_\Q : xD_1\subset D\}$.  In particular, $\LL^\sharp$  is totally isotropic of dimension $d$,
and $\mathcal{K}$ has dimension $d+1$. Moreover,  the quadratic space $\mathcal{K}/\mathcal{K}^\perp$
 is a hyperbolic plane, and so has precisely two isotropic lines.  One of them is $\LL^\sharp$, and the other is
 the subspace $\LL$ corresponding to 
 $L = \{ x\in \bm{L}_\Q : xD_1 \subset D_1 \}.$

For a general reduced $R$ of finite type, it follows from the previous paragraph
(use Exercise X.16 of \cite{Lang} and the  fact that $R$ is a Jacobson ring)  
that $\LL^\sharp$ is  a totally isotropic rank $d$ local direct  summand  of $ \Omega \otimes_k R$,
and $\mathcal{K}$ is a rank $d+1$ local direct summand.   By Lemma \ref{easyperp} there is 
a unique  totally isotropic rank $d$ local direct  summand  $\LL\not=\LL^\sharp$ of $ \Omega \otimes_k R$
contained in $\mathcal{K}$.  
As $\mathscr{N}_\Lambda$ is itself reduced and locally of finite type, 
the  construction $(G,\iota, \lambda,\varrho) \mapsto \LL$   defines a morphism of $k$-schemes
\[
\mathscr{N}_\Lambda \to \mathrm{OGr}(d)
\]
 inducing the desired bijection $ \mathscr{N}_\Lambda(k)\iso \mathscr{X}_\Lambda(k)$ on $k$-valued points. 
As the arguments of Section \ref{ss:dieudonne and special} were all done over an arbitrary extension of $k$, 
 the above morphism   induces a bijection  $ \mathscr{N}_\Lambda(k')\iso \mathscr{X}_\Lambda(k')$ for every  field extension $k'/k$.
The morphism   $\mathscr{N}_\Lambda \to \mathscr{X}_\Lambda$   is therefore 
  birational,  quasi-finite, and proper   (by Proposition \ref{prop:proj}).  
As $\mathscr{X}_\Lambda$ is smooth (and therefore normal),  
  Zariski's main theorem implies  $\mathscr{N}_\Lambda \iso \mathscr{X}_\Lambda$.  
  The claim about connected components is obvious.
\end{proof}


\subsection{The main results}
\label{ss:main}


Now we state our main results about the structure of the underlying reduced subscheme 
$\mathscr{N}_\mathrm{red}=\mathscr{N}_\mathrm{red}^+ \sqcup \mathscr{N}_\mathrm{red}^-$ of $\mathscr{N}$.
Recall from Section \ref{ss:vertex} that  $\mathscr{N}^\pm_\mathrm{red}$  is covered by the closed subschemes 
$\mathscr{N}_\Lambda^\pm$ as $\Lambda$ runs over the vertex lattices of type $t_\Lambda=2d_\Lambda\in \{2,4,6\}$ 
in the $6$-dimensional $\Q_p$-quadratic space $\bm{L}_\Q^\Phi$,
and that their intersections are given by the simple rule
\[
\mathscr{N}_{\Lambda_1}^\pm \cap \mathscr{N}_{\Lambda_2}^\pm = 
\begin{cases}
\mathscr{N}_{ \Lambda_1 \cap \Lambda_2}^\pm & \mbox{if $\Lambda_1\cap \Lambda_2$ is a vertex lattice} \\
\emptyset & \mbox{otherwise}
\end{cases}
\]
where, as before, the left hand side is understood to mean the reduced scheme underlying the scheme-theoretic intersection.
In other words, the combinatorics of the intersections are controlled by the combinatorics of the simplicial 
complex $\VV$ of Section \ref{ss:building}.

\begin{theorem}
The $k$-variety $\mathscr{N}_\Lambda^\pm$ is projective, smooth, and irreducible of dimension $d_\Lambda-1$.
   Moreover,
\begin{enumerate}
\item
if $d_\Lambda=1$ then $\mathscr{N}_\Lambda^\pm$ is  single point,
\item
if $d_\Lambda=2$ then $\mathscr{N}_\Lambda^\pm$ is isomorphic to  $\mathbb{P}^1$,
\item
if $d_\Lambda=3$ then $\mathscr{N}_\Lambda^\pm$ is isomorphic to the Fermat hypersurface 
\[
x_0^{p+1}+x_1^{p+1}+x_2^{p+1}+x_3^{p+1} = 0.
\]
\end{enumerate}
\end{theorem}

\begin{proof}
Combine Theorem \ref{thm:comparison} with the discussion of Section \ref{ss:DL special cases}.
\end{proof}

\begin{theorem}
Under the isomorphism $\mathscr{X}_\Lambda^\pm \iso \mathscr{N}_\Lambda^\pm$, the stratification of Proposition  \ref{prop:stratification}
and the stratification \[ \mathscr{N}_\Lambda^\pm 
= \biguplus_{\Lambda'\subset \Lambda} \mathscr{N}_{\Lambda'}^{\pm \circ}\]
of Section \ref{ss:vertex}   are related by
\begin{equation}\label{eq:strata match}
X_{P_r}(w_r^\pm) \iso \biguplus_{ \substack{ \Lambda'\subset \Lambda \\   d_{\Lambda'}  = r+1 }} \mathscr{N}_{ \Lambda ' }^{\pm \circ} 
\end{equation}
for all $0\le r \le d_\Lambda-1$.  In particular (by taking $r=d_\Lambda-1$),  
the dense open subvariety $\mathscr{N}_\Lambda^{\pm \circ}$ 
 is isomorphic to the Deligne-Lusztig variety $X_B(w^\pm)$ associated with a Coxeter element.
\end{theorem}

\begin{proof}
For each special lattice $L$ we defined, in Proposition \ref{prop:zink}, a sequence of lattices
\[
L=L^{(0)} \subsetneq L^{(1)} \subsetneq \cdots\subsetneq L^{(d)} =L^{(d+1)}
\]
by
$
L^{(r)}= L + \Phi(L) + \cdots + \Phi^r(L),
$
and defined a type $2d$ vertex lattice \[\Lambda_L = \{ x\in L^{(d)} : \Phi(x) = x \}.\]  The bijection (\ref{special bijection}) identifies $\mathscr{N}_\Lambda^\pm(k)$ 
with the set of special lattices $L$ with $\Lambda_L\subset \Lambda$, and the $k$-points of the right hand side of
(\ref{eq:strata match}) correspond to those $L$ for which $\Lambda_L$ has type $2r+2$; in other words, those $L$ for which
\[
L=L^{(0)} \subsetneq L^{(1)} \subsetneq \cdots\subsetneq L^{(r+1)} =L^{(r+2)}.
\]
If we instead define
$
\underline{L}^{(r)}= L \cap \Phi(L) \cap \cdots \cap \Phi^r(L),
$
this condition is equivalent to
\[
\underline{L}^{(r+2)} = \underline{L}^{(r+1)} \subsetneq \cdots \subsetneq \underline{L}^{(1)} \subsetneq\underline{L}^{(0)}=L.
\]
In the proof of Proposition \ref{prop:stratification}, this is the same as the condition defining  the strata $X_{P_r}(w_r^\pm)$.
\end{proof}

\begin{theorem}\label{thm:connected}
The reduced $k$-scheme  $\mathscr{N}_\mathrm{red}$ is equidimensional of  dimension two.
It has two connected components,  $\mathscr{N}^+_\mathrm{red}$ and $\mathscr{N}^-_\mathrm{red}$, and these connected components 
are isomorphic.   The irreducible components of  $\mathscr{N}_\mathrm{red}$ are precisely the closed subschemes
$\mathscr{N}^\pm_\Lambda$ as $\Lambda$ varies over the type $6$ vertex lattices.  
Furthermore:
\begin{enumerate}
 \item
For each irreducible component $\mathscr{N}_\Lambda$, there are exactly $(p^3+1)(p+1)$  irreducible components 
$\mathscr{N}_{\Lambda'}$ such that $\mathscr{N}_\Lambda \cap \mathscr{N}_{\Lambda'} \iso \mathbb{P}^1$,
and $(p^3+1)(p^2+1)$  irreducible components 
$\mathscr{N}_{\Lambda'}$ such that $\mathscr{N}_\Lambda \cap \mathscr{N}_{\Lambda'}$ consists of a single point.
 \item
 For each type $4$ vertex lattice $\Lambda$, the closed subscheme $\mathscr{N}_\Lambda \iso \mathbb{P}^1$ is contained in 
exactly two irreducible components, and is equal to their intersection. 
\end{enumerate}
\end{theorem}

\begin{proof}
The isomorphism $\mathscr{N}_\mathrm{red}^+\iso \mathscr{N}_\mathrm{red}^-$ follows from the isomorphism $\mathscr{N}^+\iso \mathscr{N}^-$ of Section \ref{ss:RZ}.
The connectedness of $\mathscr{N}_\mathrm{red}^\pm$ follows from Corollary \ref{cor:connected}.
The remaining claims are clear from the Theorems above and the discussion of Section \ref{ss:DL special cases}.
\end{proof}


\subsection{Hermitian vertex lattices}
\label{ss:building2}


As in \cite{RTW}, \cite{Vol}, and \cite{VolWed}, it is possible to describe the stratification of $\mathscr{N}$ in terms of the 
Bruhat-Tits building of the  special unitary group $J^\mathrm{der}$, although in our setting the description in these terms is slightly convoluted.
Recall from Remark \ref{excRemark} the central isogeny  $J^\mathrm{der}\to \SO(\bm{L}_\Q^\Phi)$. 
Using \cite[4.2.15]{BTII} we see that this gives an identification of the building $\mathcal{BT}$ of 
$\SO(\bm{L}_\Q^\Phi)$, which was described in Section \ref{ss:building},   with the building  of $J^\mathrm{der}$. 
 Therefore, using \cite{Vol} and $J^\mathrm{der}\iso \mathrm{SU}(T)$, we can see that the underlying 
simplicial complex of the building  $\mathcal{BT}$ can also be 
described using $\co_E$-lattices $\Xi$ in the split Hermitian space $T$ of dimension $4$ over $E$.

We say that an $\co_E$-lattice $\Xi \subset T$  is a \emph{Hermitian vertex lattice} if 
\[
\Xi\subset \Xi^\vee\subset p^{-1}\Xi.
\]
The \emph{type} of $\Xi$ is $\dim_{\F_{p^2}}(\Xi^\vee/\Xi)$; the type  can be $0$, $2$ or $4$. 
As in \cite{Vol}, these Hermitian vertex lattices correspond
bijectively to the vertices of the Bruhat-Tits building of $\mathrm{SU}(T)$.
The action of the group $\mathrm{SU}(T)$  preserves the vertex type and is transitive on the set of vertices of a given type. 
The simplicial structure of the building of $\mathrm{SU}(T)$ is generated, as above, using 
a notion of adjacency, in which $\Xi$ and $\Xi'$ are adjacent
if either $\Xi\subset \Xi'$ or $\Xi'\subset \Xi$.
Consider now the identification of the buildings given by the central
isogeny $\mathrm{SU}(T)\to \SO(\bm{L}_\Q^\Phi)$. We can see by looking at the local Dynkin diagrams that  Hermitian vertex
lattices  $\Xi$ of type $0$ and $4$ are sent to vertex lattices $\Lambda$ of type $6$,  and Hermitian vertex lattices $\Xi$ of type $2$ are sent to 
vertex lattices $\Lambda$ of type $2$. Note that $\SO(\bm{L}_\Q^\Phi)$ acts transitively
on the set of vertex lattices of type $6$, but the map $\mathrm{SU}(T)\to \SO(\bm{L}_\Q^\Phi)$ is not surjective 
on $\Q_p$-points:   its image is the kernel of the spinor norm.

 Consider the set $\SS$ which is defined as the disjoint union of the set of Hermitian vertex lattices $\Xi$ 
with the set of all pairs $\{\Xi,\Xi'\}$ consisting of adjacent Hermitian vertex lattices of types $0$ and $4$.
Note that there is a natural bijection between the set $\SS$ and the set of all vertex lattices $\Lambda$.
Hermitian vertex lattices of type $0$ and $4$ in $\SS$ correspond to vertex lattices of type $6$,
Hermitian vertex lattices of type $2$ in $\SS$ correspond to vertex lattices of type $2$, and finally the  pairs $\{\Xi,\Xi'\}$ correspond 
to vertex lattices of type $4$.

 We define a partial order on  $\SS$ as follows.
For two Hermitian vertex lattices we define $\Xi < \Xi'$ if  either 
\begin{enumerate}
\item
$\Xi$ is of type $2$, $\Xi'$ is of type $0$, and  $\Xi\subset \Xi'$; or
\item
$\Xi$ is type $2$,  $\Xi'$ is of type $4$, and $\Xi'\subset \Xi$
\end{enumerate}
(so Hermitian vertex lattices of type $0$ and $4$ are not comparable). Two  pairs $\{\Xi_1,\Xi_1'\}$ and 
$\{\Xi_2,\Xi_2'\}$ in $\SS$ are not compared. If $\Xi$ is a Hermitian vertex lattice then $\Xi < \{ \Xi_1,\Xi_2\}$
if $\Xi$, $\Xi_1$, and $\Xi_2$ form a simplex in the building of $J^\mathrm{der}$ (which requires that $\Xi$ have type $2$).
Finally, $\{ \Xi_1,\Xi_2\}<\Xi$ if $\Xi \in \{ \Xi_1,\Xi_2\}$.   Under the bijection between $\SS$ and the set of vertex lattices this
partial order corresponds to  inclusion of vertex lattices.  Define an adjacency relation in $\SS$ by
$x\sim_{\SS} y$ if  either $x< y$ or $y<x$.  We also define a dimension function $d: \SS\to \{0,1,2\}$ by $d(x)=0$ if $x$ is a 
Hermitian vertex lattice of type $2$, $d(x)=2$ if $x$ is a Hermitian vertex lattice of type $0$ or $4$, and $d(x)=1$ if $x$ is a pair
$\{\Xi,\Xi'\}$.

 The following theorems are simply  restatements in this new language of some results of the previous subsection.

\begin{theorem}
Writing the reduced $k$-scheme as a union
\[
\mathscr{M}_\mathrm{red}=\biguplus\nolimits_{\ell\in \Z} \mathscr{M}^{(\ell)}_\mathrm{red}
\]
gives the decomposition of 
$\mathscr{M}_\mathrm{red}$ into its connected components $\mathscr{M}^{(\ell)}_\mathrm{red}$.  These connected components 
are all isomorphic and are of pure dimension $2$. 
\begin{enumerate}
\item
There is a stratification of $\mathscr{M}^{(0)}_\mathrm{red}$ by locally closed 
smooth subschemes given by 
\[
\mathscr{M}^{(0)}_\mathrm{red}=\biguplus_{ x\in\SS} \mathscr{M}^\circ_x.
\]
Each stratum $ \mathscr{M}^\circ_x$ is isomorphic to $\mathscr{N}^{+\circ}_\Lambda$,
where $\Lambda$ is the vertex lattice   that corresponds to $x$, and is therefore isomorphic to a Deligne-Lustzig variety of dimension $d(x)$. 
The closure  ${\mathscr{M}_x}$ of any $ \mathscr{M}^\circ_x$ in 
$\mathscr{M}^{(0)}_\mathrm{red}$ is 
\[
  {\mathscr{M}_x}=\biguplus_{y \leq x} \mathscr{M}^\circ_y.
\]

\item
We have $\mathscr{M}_y\subset  \mathscr{M}_x$ if and only if $y\leq x$.
In particular,  the irreducible components of  $\mathscr{M}^{(0)}_\mathrm{red}$ are precisely the closed subschemes 
$\mathscr{M}_\Xi$  for $\Xi\in\SS$ a Hermitian vertex lattice of type $0$ or $4$.

\item
 The schemes $\mathscr{M}_x$, are as follows:
\begin{enumerate}
\item
if $d(x)=0$ then $\mathscr{M}_x$ is  single point,
\item
if $d(x)=1$ then $\mathscr{M}_x$ is isomorphic to  $\mathbb{P}^1$,
\item
if $d(x)=2$ then $\mathscr{M}_x$ is isomorphic to the Fermat hypersurface 
\end{enumerate}
\[
x_0^{p+1}+x_1^{p+1}+x_2^{p+1}+x_3^{p+1} = 0.
\]
\end{enumerate}
\end{theorem}

\begin{theorem}
The irreducible  components of $\mathscr{M}^{(0)}_\mathrm{red}$ 
are parametrized by vertices of type $0$ and $4$ in the Bruhat-Tits building of $J^\mathrm{der}$.
Two irreducible components $\mathscr{M}_\Xi$ and $\mathscr{M}_{\Xi'}$
intersect if and only if either $\Xi$ and $\Xi'$ are either adjacent, or are adjacent to   common element of $\SS$.
If they are adjacent then one is type $0$, the other of type $4$, and   they intersect along 
a $\mathbb{P}^1$. If they are not adjacent but  have a common adjacent point  $y\in \SS$, then
$y$ is a Hermitian vertex lattice of type $2$, and $\mathscr{M}_{\Xi} \cap \mathscr{M}_{\Xi'} = \mathscr{M}_y$ is a single point.
\end{theorem}


\section{Applications to Shimura varieties}
\label{s:global}


In this section we use our  explicit description of the Rapoport-Zink space $\mathscr{N} = p^\Z\backslash \mathscr{M}$
to describe the supersingular locus of a $\GU(2,2)$-Shimura variety.  
With  the results of Section \ref{ss:main} in hand, this is exactly as in the $\GU(n-1,1)$ cases studied in 
\cite{RTW} and \cite{VolWed}.  Accordingly, our discussion will be brief.


\subsection{The Shimura variety}


Let   $E\subset \C$  be a quadratic imaginary field, fix a prime $p>2$ inert in $E$,
and let $\co\subset E$ be the integral closure of   $\Z_{(p)}$ in $E$.
Let $V$ be a free $\co$-module of rank $4$ endowed with a perfect $\co$-valued Hermitian form
$\langle\cdot,\cdot\rangle$ of signature $(2,2)$, and denote by $G=\GU(V)$ the group of unitary similitudes
of $V$.  It is a reductive group over $\Z_{(p)}$.   Fix a compact open subgroup $U^p \subset G(\A_f^p)$, and define
$U_p=G(\Z_p)$ and $U=U_pU^p \subset G(\A_f)$.

The Grassmannian $\mathcal{D}$ of negative definite planes in $V\otimes_{\co}  \C$ is a smooth 
complex manifold of dimension $4$, with an action of $G(\R)$.  Define
\[
M_U(\C) = G(\Q) \backslash  (  \mathcal{D} \times G(\A_f) / U ).
\]
For sufficiently small $U^p$, this is a smooth complex manifold parametrizing prime-to-$p$ isogeny classes of 
quadruples $(A, \iota, \lambda , [\eta^p])$,  in which $A$ is an abelian variety of dimension $4$, 
$\iota : \co \to \End(A)_{(p)}$ is a ring homomorphism such that 
\[
\det( T - \iota(\alpha) ; \Lie(A) ) = (T-\alpha)^2 ( T-\overline{\alpha})^2
\]
for all $\alpha\in \co$, $\lambda \in \Hom(A,A^\vee)_{(p)}$  a prime-to-$p$-quasi-polarization 
satisfying 
\[
\lambda \circ \iota (\overline{\alpha}) = \iota (\alpha)^\vee \circ \lambda
\]
for all $\alpha\in \co$, and 
$[\eta^p]$ is the $U^p$-orbit of an $\co \otimes \A_f^p$-linear isomorphism
\[
\eta^p:   \widehat{\mathrm{Ta}}^p(A)\otimes \A_f^p \to V \otimes \A_f^p  
\]
respecting the Hermitian forms up to scaling by 
$(\A_f^p)^\times$ (the Hermitian form on the source is determined by $\lambda$, as in (\ref{hermite})).
A \emph{prime-to-$p$-isogeny} between two such pairs $(A, \iota, \lambda , [\eta^p])$ and 
$(A', \iota', \lambda' , [\eta^{p\prime} ] )$  is an $\co$-linear quasi-isogeny in $\Hom(A,A')_{(p)}$
of degree prime to $p$ that respects the level structures, and such that $\lambda'$ pulls back to a $\Z_{(p)}^\times$-multiple 
of $\lambda$.

The parametrization is similar to the constructions found in \cite{KR}, and can be described as follows.
For each triple $(A,i,\lambda,[\eta])$  above,
the existence of $\eta^p$ implies that $H_1(A,\Q)$ and $V\otimes \Q$
are isomorphic, as Hermitian spaces, locally at all places  $v\nmid p$.  
But this implies  that they are also isomorphic at $p$, and hence there is a global isomorphism 
\[
\beta: H_1(A,\Q) \to V\otimes \Q.
\]
As $\mathrm{Ta}_p(A) \otimes \Q_p \iso V \otimes \Q_p$,  a  result of Jacobowitz, stated  in \cite[Proposition 2.14]{KR}, shows that there is 
a unique $U_p$-orbit of isomorphisms $\mathrm{Ta}_p(A) \iso V \otimes \Z_p$ compatible with the $\co$-actions 
and Hermitian forms.  Thus there is a 
unique way to extend $\eta^p$ to a $U$-orbit of isomorphisms
\[
\eta: \widehat{\mathrm{Ta}}(A)\otimes \A_f \iso V \otimes \A_f
\] 
compatible with the  $\co$-actions and the symplectic forms, and identifying $\mathrm{Ta}_p(A)$ with $V \otimes \Z_p$.  The composition
 \[
 V\otimes \A_f \map{\eta^{-1} } H_1(A,\A_f) \map{\beta} V\otimes\A_f
 \]
 defines an element $g\in G(\A_f)/U$, and the Hodge structure on $V\otimes\R$ induced by the 
 isomorphism $\beta$ corresponds to a point of $\mathcal{D}$, as in \cite[Section 3]{KR}.


\subsection{The uniformization theorem}


Let $k$ be an algebraic closure of the field of $p$ elements.

Extending the moduli problem of the previous subsection to $\Z_{(p)}$-schemes in the obvious way yields a scheme $M_U$,
smooth of relative dimension $4$  over $\Z_{(p)}$.  Denote by $M_U^\mathrm{ss}$  the reduced supersingular locus 
of the geometric special fiber $M_U \times_{\Z_{(p)} }k$. A choice of geometric point $(\bm{A},\bm{\iota},  \bm{\lambda} , [\bm{\eta}] ) \in M_U^\mathrm{ss}(k)$
determines  a base point $(\bm{G},\bm{\iota},\bm{\lambda})$ with $\bm{G}=\bm{A}[p^\infty]$, and so 
defines a Rapoport-Zink space $\mathscr{M}$ as in Section \ref{ss:RZ}, endowed with an action of the subgroup $J\subset \End(\bm{G})_\Q^\times$.
Denote by $I(\Q) \subset \End(A)_\Q^\times$ the subgroup of $\co$-linear quasi-automorphisms that preserve
the $\Q^\times$-span of $\bm{\lambda}$.  It is the group of $\Q$-points of an algebraic group $I$ over $\Q$
satisfying  $I(\Q_p)\iso J$, and the orbit $[\bm{\eta}]$ determines  a right $U^p$-orbit of isomorphisms $I(\A_f^p) \iso G(\A_f^p)$.
In particular, $I(\Q)$ acts on both $\mathscr{M}$ and on $G(\A_f^p)/U^p$.

\begin{theorem}[Rapoport-Zink]
There is an isomorphism of $k$-schemes
\[
M_U^\mathrm{ss}  \iso  I(\Q) \backslash ( \mathscr{M}_\mathrm{red} \times G(\A_f^p) /U^p).
\]
\end{theorem}

As in \cite[Corollary 6.2]{Vol}, combining the above uniformization theorem with the results of Section \ref{ss:main}
yields the following corollary.

\begin{corollary}
The $k$-scheme $M_U^\mathrm{ss}$ has pure dimension $2$.  For $U^p$ sufficiently small, all irreducible components of 
$M_U^\mathrm{ss}$ are isomorphic to the Fermat hypersurface 
\[
x_0^{p+1}+x_1^{p+1}+x_2^{p+1} + x_3^{p+1}=0,
\]
and any two irreducible components intersect either trivially, 
intersect at a single point, or their intersection is isomorphic to $\mathbb{P}^1$. 
Here ``intersection" is understood to mean the reduced scheme underlying the scheme-theoretic intersection.
\end{corollary}

\medskip

\bibliographystyle{plain}

\end{document}